\documentclass[11pt, twoside]{article}

\usepackage{amssymb}
\usepackage{amsmath}
\usepackage{amsthm}
\usepackage{color}

\usepackage{indentfirst}

\usepackage{txfonts}

\allowdisplaybreaks

\def\va{\vec{a}}
\def\vb{\vec{b}}
\def\vc{\vec{c}}
\def\vt{\vec{\beta}}
\def\vf{\vec{\alpha}}
\def\vp{\vec p}
\def\vq{\vec{q}}
\def\vg{\vec \gamma}
\def\lv{{L^{\vp}_{\vf}}}
\def\lt{{L^{\vq}_{\vt}}}

\def\rr{{\mathbb R}}

\def\cn{{\mathbb{C}^n}}

\def\bn{{\mathbb{B}_n}}

\def\fz{\infty }
\def\az{\alpha}
\def\bz{\beta}

\def\vaz{\varepsilon}

\def\lf{\left}
\def\r{\right}

\def\hs{\hspace{0.35cm}}
\def\ls{\lesssim}
\def\gs{\gtrsim}

\def\noz{\nonumber}

\def\XXint#1#2#3{{\setbox0=\hbox{$#1{#2#3}{\int}$ }
\vcenter{\hbox{$#2#3$ }}\kern-.6\wd0}}

\DeclareMathOperator{\esssup}{ess\,sup}

\def\f{\frac}

\def\({\left(}
\def \){ \right)}

\newtheorem{theorem}{Theorem}[section]
\newtheorem{lemma}[theorem]{Lemma}

\newtheorem{proposition}[theorem]{Proposition}
\newtheorem{observation}[theorem]{Observation}

\newtheorem{question}[theorem]{Question}
\theoremstyle{definition}

\renewcommand{\appendix}{\par
   \setcounter{section}{0}%
   \setcounter{subsection}{0}%
   \setcounter{subsubsection}{0}%
   \gdef\thesection{\@Alph\c@section}%
   \gdef\thesubsection{\@Alph\c@section.\@arabic\c@subsection}%
   \gdef\theHsection{\@Alph\c@section.}%
   \gdef\theHsubsection{\@Alph\c@section.\@arabic\c@subsection}%
   \csname appendixmore\endcsname
 }

\numberwithin{equation}{section}

\begin{document}

\arraycolsep=1pt

\title{\bf\Large
$L^{\vec{p}}-L^{\vec{q}}$ Boundedness of Multiparameter Forelli-Rudin Type
Operators on the Product of Unit Balls of $\mathbb{C}^n$
\footnotetext{\hspace{-0.35cm} 2020 {\it
Mathematics Subject Classification}. Primary 42B35;
Secondary 42B99, 47B38, 32A36.
\endgraf {\it Key words and phrases.}
Mixed-norm Lebesgue space,
$L^{\vec p}-L^{\vec q}$ boundedness,
Schur's test, Forelli-Rudin type operator,
Bergman projection,
Berezin transform.
\endgraf
The first author would like to thank Prof. Dachun Yang for some helpful suggestions on 
this article. The authors would like to thank Prof. Ruhan Zhao and Prof. Kehe Zhu for 
several comments on Schur's tests and some nice suggestions. 
The authors would also like to thank the referee for several insightful comments.
L. Huang is supported by the National Natural Science Foundation of China (Grant No. 12201139)
and Guangdong Basic and Applied Basic Research Foundation (Grant No. 2021A1515110905).
X. Wang is supported by the National Natural Science Foundation of China (Grant No. 11971125).}}
\author{Long Huang, Xiaofeng Wang\footnote{Corresponding author,
E-mail: \texttt{wxf@gzhu.edu.cn}}\ \
and Zhicheng Zeng
}
\date{}
\maketitle

\vspace{-0.8cm}

\begin{center}
\begin{minipage}{13cm}
{\small {\bf Abstract}\quad
In this work, we provide a complete characterization of the boundedness of two classes of multiparameter Forelli-Rudin type operators from one mixed-norm Lebesgue space $L^{\vec p}$ to another space $L^{\vec q}$, when $1\leq \vec{p}\leq \vec q<\infty$, equipped with possibly different weights. Using these characterizations, we establish the necessary and sufficient conditions for both $L^{\vec p}-L^{\vec q}$ boundedness of the weighted multiparameter Berezin transform and $L^{\vec p}-A^{\vec q}$ boundedness of the weighted multiparameter Bergman projection, where $A^{\vec q}$ denotes the mixed-norm Bergman space.
Our approach presents several novelties. Firstly, we conduct refined integral estimates of holomorphic functions on the unit ball in $\mathbb{C}^n$. Secondly, we adapt the classical Schur's test to different weighted mixed-norm Lebesgue spaces. These improvements are crucial in our proofs and allow us to establish the desired characterization and sharp conditions.
}
\end{minipage}
\end{center}


\vspace{0.2cm}

\section{Introduction\label{s0}}

As a natural generalization of the classical Lebesgue space $L^p$,
the mixed-norm Lebesgue space $L^{\vec p}$ was first introduced
and studied by Benedek and Panzone \cite{bp61} in 1961.
After that, numerous other function spaces with
mixed norms were introduced and studied.
For instance, Chen-Sun \cite{cs} studied the
mixed-norm (iterated) weak Lebesgue spaces; 
Cleanthous-Georgiadis \cite{cg19,cg22} investigated the mixed-norm modulation spaces and product modulation spaces; Cleanthous et al. \cite{cgn17} introduced 
anisotropic mixed-norm Hardy spaces and then the first author and his collaborators  
further considered various real-variable characterizations and their applications
about these mixed-norm Hardy spaces in \cite{hcy,hlyy,hlyy18,hyy21};
Cleanthous et al. \cite{cgn19,cgn19-1}, Georgiadis et al. \cite{gn16,gjn17} as well as the first author and his collaborators \cite{hlyy20} focused on mixed-norm Triebel-Lizorkin or Besov spaces.
Moreover, since the mixed-norm Lebesgue space has finer structures
than the classical one, and the wider usefulness of mixed-norm function spaces
within the context of harmonic analysis, partial differential equations, and geometric inequalities,
they have attached considerable interest (see, for instance,
\cite{bm16,bm17,bm22,cs20,cs22,cgn19-2,hwyy23,s67}).

This article focuses on the action of two classes of multiparameter Forelli-Rudin type operators on weighted mixed-norm Lebesgue spaces, which are natural extensions of the well-known Forelli-Rudin type operators originally introduced by Forelli and Rudin \cite{fr74} in 1974. Additionally, we address the characterizations for the boundedness of multiparameter Forelli-Rudin type operators for critical points.
Our proofs present several novelties. Firstly, we conduct refined integral estimates of holomorphic functions on the unit ball in $\mathbb{C}^n$. Secondly, we adapt the classical Schur's test to different weighted mixed-norm Lebesgue spaces. These improvements are crucial in our proofs and allow us to establish the desired characterization. As applications, we also establish the necessary and sufficient conditions for both $L^{\vec p}-L^{\vec q}$ boundedness of the weighted multiparameter Berezin transform and $L^{\vec p}-A^{\vec q}$ boundedness of the weighted multiparameter Bergman projection, where $A^{\vec q}$ denotes the mixed-norm Bergman space. Our motivation of this work comes from the following classical results in the field.

Let $n$ be a fixed positive integer and
$\mathbb{C}^n$ the vector space of all ordered
$n$-tuples $z:=(z_1,\ldots,z_n)$ of complex numbers,
with inner product $\langle z,\omega\rangle:=z_1\overline{\omega}_1+
\cdots+z_n\overline{\omega}_n$ and norm $|z|:=\langle z,z\rangle^{\frac 12}.$
Let $\mathbb{B}_n:=\{z\in \mathbb{C}^n:\
|z|<1\}$ be the open unit ball in $\mathbb{C}^n$ and
$\mathbb{S}_n:=\{z\in \mathbb{C}^n:\ |z|=1\}$ denote its boundary.
For given $a,b,c\in \mathbb{R}$, the following general
Forelli-Rudin type operators, introduced by Kurens-Zhu \cite{kz06},
are defined by
\begin{align*}
T_{a,\,b,\,c}f(z)=(1-|z|^2)^{a}\int_{\mathbb{B}_n}\frac{(1-|u|^2)^{b}}{(1-\langle z,u \rangle)^{c}}f(u)\,d\upsilon(u)
\end{align*}
and
\begin{align*}
S_{\negthinspace a,\,b,\,c}f(z)=(1-|z|^2)^{a}\int_{\mathbb{B}_n}\frac{(1-|u|^2)^{b}}{|1-\langle z,u \rangle|^{c}}f(u)\,d\upsilon(u),
\end{align*}
where $d\upsilon$ is the normalized volume measure on
$\mathbb{B}_n$ satisfying $\upsilon(\mathbb{B}_n)=1$.
Recall that the history of study for the special case of these operators can be traced back to
the work of Stein \cite{s73}, in which Stein proved that the operator
$T_{0,\,0,\,n+1}$ is bounded on $L^p(\mathbb{B}_n)$ for $1<p<\fz$.
In 1974, Forelli-Rudin \cite{fr74} then proved that $T_{0,\,\sigma+it,\,n+1+\sigma+it}$ is bounded on $L^p(\mathbb{B}_n)$ if and only if $$(\sigma+1)p>1,$$ where $1\le p<\fz$, $\sigma>-1$, $t\in \rr$,
and $i=\sqrt{-1}$.
Later, Kolaski \cite{k79} further investigated the operator
$T_{0,\,\sigma+it,\,n+1+\sigma+it}$ from the Bergman projection perspective.
In 1991, Zhu \cite{z91} showed that the necessary and sufficient condition
of boundedness of the general operator $T_{a,\,b,\,c}$ on
$L^p(d A_{\lambda})$, with
$c=n+1+a+b$ when $n=1$, is $$-pa<\lambda+1<p(b+1),$$
where $1\le p<\fz$ and $dA_\lambda(z):=(1-|z|^2)^{\lambda}\,dA(z)$
with $\lambda\in\rr$
and $dA$ being the normalized area measure in the open unit disc
of $\mathbb{C}$. This result nowadays has been extended to the
high-dimensional case by Kures-Zhu \cite{kz06} for all positive
numbers $c$. On the other hand, the work \cite{z15,z23} further determined
the boundedness of operators $T_{0,\,b,\,c}$ and $S_{\negthinspace 0,\,b,\,c}$ from $L^p$ spaces to $L^q$ spaces with $1\le p\le q<\fz$. For more results in this direction, we refer to \cite{hw23,ku19,qwg23,qwg23',zz23}. 
For compactness of Forelli-Rudin type operators, we refer to \cite{dw22},
in which Ding and Wang gave some compactness characterizations of operators $T_{0,\,0,\,c}$ and $S_{0,\,0,\,c}$.
Very recently,
Zhu \cite{z22} extended several existing boundedness results of operators $T_{a,\,b,\,c}$ using 
embeddings of Bergman spaces and fractional radial differential operators, and further characterized
when the embedding of one weighted Bergman space into another is compact.

These two operators have gained significant attention due to their connection to some important operators in analysis and operator theory. For instance, the holomorphic Bergman projection on $\bn$ is a special case of the Forelli-Rudin type operator $T_{a,\,b,\,c}$ with $a=b=0$ and $c=n+1$. The Berezin transform on $\bn$ is a special case of the Forelli-Rudin type operator $S_{\negthinspace a,\,b,\,c}$ with $a=n+1$, $b=0$, and $c=2(n+1)$. It is well known that the Bergman projection and Berezin transform play a fundamental role in the theory of holomorphic function spaces, and have been extensively studied in the context of Toeplitz and Hankel operators, as well as other related areas (see \cite{zbook05,zbook07,z21}).

On the other hand, multiparameter operators have also been a significant research area in analysis and operator theory. Many authors have studied the boundedness of multipliers and singular integral operators on multiparameter function spaces over the past few decades. Examples include Ricci-Stein's work on multiparameter singular integrals and maximal functions \cite{rs92}, M\"{u}ller-Ricci-Stein's study of Marcinkiewicz multipliers and multiparameter structure on Heisenberg groups \cite{mrs95,mrs96}, and Fefferman-Pipher's results on multiparameter operators and sharp weighted inequalities \cite{fp97}. Building on the insights and also the aforementioned conclusions on classical Forelli-Rudin type operators, we extend $T_{a,b,c}$ and $S_{\negthinspace a,b,c}$ to the multiparameter cases in this article.
To be exact, we focus on the following two classes of multiparameter operators
\begin{align*}
T_{\va,\,\vb,\,\vc}f(z,w)=(1-|z|^2)^{a_1}(1-|w|^2)^{a_2}
\int_{\mathbb{B}_n}\int_{\mathbb{B}_n}
\frac{(1-|u|^2)^{b_1}(1-|\eta|^2)^{b_2}}{(1-\langle z,u \rangle)^{c_1}(1-\langle w,\eta \rangle)^{c_2}}f(u,\eta)\,d\upsilon(u)\,d\upsilon(\eta)
\end{align*}
and
\begin{align*}
S_{\negthinspace\va,\,\vb,\,\vc}f(z,w)=(1-|z|^2)^{a_1}(1-|w|^2)^{a_2}
\int_{\mathbb{B}_n}\int_{\mathbb{B}_n}\frac{(1-|u|^2)^{b_1}(1-|\eta|^2)^{b_2}}{|1-\langle z,u \rangle|^{c_1}|1-\langle w,\eta \rangle|^{c_2}}f(u,\eta)\,d\upsilon(u)\,d\upsilon(\eta),
\end{align*}
here and thereafter, vectors $\va:=(a_1,a_2)$, $\vb:=(b_1,b_2)$, $\vc:=(c_1,c_2)\in \rr^2$. Obviously, these two integral operators are natural extensions of the operators
$T_{a,\,b,\,c}$ and $S_{\negthinspace a,\,b,\,c}$. In addition, we would like to clarify that we limit our focus to the two-dimensional case of the parameters $\va,\ \vb$, and $\vc$ in this article. However, the higher dimensional case can be approached in a similar manner.
For these multiparameter Forelli-Rudin type operators, we
are interested in the following questions:

\begin{question}\label{1q1}
	On which spaces are these operators bounded?
	Are there two distinct spaces $X$ and $Y$ such that operators $T_{\va,\,\vb,\,\vc}$ and $S_{\negthinspace \va,\,\vb,\,\vc}$ are bounded from $X$ to $Y$?
\end{question}

\begin{question}\label{1q2}
If such spaces $X$ and $Y$ exist, what are the necessary and sufficient conditions that ensure the boundedness of $T_{\va,\,\vb,\,\vc}$ and $S_{\negthinspace \va,\,\vb,\,\vc}$?
\end{question}

In this article, we not only address Questions \ref{1q1} and \ref{1q2}, but we also provide a precise determination of the conditions under which the operators $T_{\va,\,\vb,\,\vc}$ and $S_{\negthinspace \va,\,\vb,\,\vc}$ are bounded. Specifically, we demonstrate that the class of weighted mixed-norm Lebesgue spaces provides a suitable framework for analyzing the behavior of the Forelli-Rudin type operators $T_{\va,\,\vb,\,\vc}$ and $S_{\negthinspace \va,\,\vb,\,\vc}$\,.
Notice that, for any given $\vp:=(p_1,p_2)\in[1,\fz]^2$ and
$\vf:=(\alpha_1,\alpha_2)\in(-1,\fz)^2$,
the weighted mixed-norm Lebesgue space $\lv:=L^{\vp}(\bn\times\bn,d\upsilon_{\alpha_1}d\upsilon_{\alpha_2})$
is defined to be the set of all
measurable functions $f$ such that
$$\|f\|_{\lv}:=\lf\{\int_{\bn}\lf[\int_{\bn}|f(z,w)|^{p_1}
\,d\upsilon_{\alpha_1}(z)\r]^{\f{p_2}{p_1}}\, d\upsilon_{\alpha_2}(w)\r\}^{\f{1}{p_2}}<\fz$$
with the usual modifications made when $p_i=\fz$
for some $i\in \{1,2\}$ (see \cite{bp61,cs20,cg19,cg22,hlyy,hlyy18} for the unweighted mixed-norm Lebesgue spaces
and their applications). Here and thereafter, for any $\theta\in(-1,\fz)$,
$d\upsilon_\theta(z):=c_\theta(1-|z|^2)^\theta\,d\upsilon(z)$
and $$ c_\theta:=\frac{\Gamma(n+\theta+1)}{n!\Gamma(\theta+1)}$$ is the
normalized constant such that $\upsilon_\theta(\bn)=1$, where $\Gamma$ denotes the Gamma function.
Clearly, when $p_1=p_2=p$ and $\alpha_1=\alpha_2=\alpha$,
the space $\lv$ then goes back to the weighted Lebesgue space $L^p_\alpha$.

With the notion of mixed-norm Lebesgue spaces, our main results are stated as follows.
In the rest of this article,
we always suppose $\vf:=(\alpha_1,\alpha_2)\in(-1,\fz)^2$
and $\vec{\beta}:=(\beta_1,\beta_2)\in(-1,\fz)^2$.

\begin{theorem}\label{0t1}
Let $\vp:=(p_1,p_2)$ and $\vq:=(q_1,q_2)$
satisfy $1<p_-\le p_+\le q_-<\fz$ with
$p_+:=\max\{p_1,p_2\}$, $p_-:=\min\{p_1,p_2\}$, and $q_-:=\min\{q_1,q_2\}$.
Then the following statements are equivalent.
\begin{enumerate}
\item[{\rm(i)}]
The operator $S_{\negthinspace \va,\,\vb,\,\vc}$ is
bounded from $\lv$ to $\lt$\,.
\vspace{-0.2cm}
\item[{\rm(ii)}]
The operator $T_{\va,\,\vb,\,\vc}$ is bounded from $\lv$ to $\lt$\,.
\item[{\rm(iii)}]
The parameters satisfy that, for any $i\in\{1,2\}$,
\begin{align*}
   \left\{
   \begin{aligned}
   &-q_i a_i<\beta_i+1,\ \ \alpha_i+1<p_i(b_i+1),\\
   &c_i\le n+1+a_i+b_i+\frac{n+1+\beta_i}{q_i}-
   \frac{n+1+\alpha_i}{p_i}.
   \end{aligned}
   \right.
\end{align*}
\end{enumerate}
\end{theorem}

We also establish the following three $\lv-\lt$ boundedness characterizations for
the endpoints of exponents $\vp:=(1,p_2)$, $\vp:=(p_1,1)$, and $\vp:=(1,1)$, respectively,
where $1<p_1,\, p_2<\fz$.

\begin{theorem}\label{0t2}
Let $\vp:=(1,p_2)$ and $\vq:=(q_1,q_2)$
satisfy $1<p_2\le q_-<\fz$ with
$q_-:=\min\{q_1,q_2\}$.
Then the following statements are equivalent.
\begin{enumerate}
\item[{\rm(i)}]
The operator $S_{\negthinspace \va,\,\vb,\,\vc}$ is
bounded from $\lv$ to $\lt$\,.
\vspace{-0.2cm}
\item[{\rm(ii)}]
The operator $T_{\va,\,\vb,\,\vc}$ is bounded from $\lv$ to $\lt$\,.
\item[{\rm(iii)}]
The parameters satisfy that, for any $i\in\{1,2\}$,
\begin{align*}
     \left\{
     \begin{aligned}
     &-q_i a_i<\beta_i+1,\ \ \alpha_1=b_1,\ \ c_1< a_1+\frac{n+1+\beta_1}{q_1}, \\
     &\alpha_2+1<p_2(b_2+1),\ \ c_2\le n+1+a_2+b_2+\frac{n+1+\beta_2}{q_2}-
     \frac{n+1+\alpha_2}{p_2}.
     \end{aligned}
     \right.
     \end{align*}
\end{enumerate}
\end{theorem}

\begin{theorem}\label{0t3}
Let $\vp:=(p_1,1)$ and $\vq:=(q_1,q_2)$
satisfy $1<p_1\le q_-<\fz$ with
$q_-:=\min\{q_1,q_2\}$.
Then the following statements are equivalent.
\begin{enumerate}
\item[{\rm(i)}]
The operator $S_{\negthinspace \va,\,\vb,\,\vc}$ is
bounded from $\lv$ to $\lt$\,.
\vspace{-0.2cm}
\item[{\rm(ii)}]
The operator $T_{\va,\,\vb,\,\vc}$ is bounded from $\lv$ to $\lt$\,.
\item[{\rm(iii)}]
The parameters satisfy that, for any $i\in\{1,2\}$,
\begin{align*}
     \left\{
     \begin{aligned}
     &\alpha_1+1<p_1(b_1+1),\ \ c_1\le n+1+a_1+b_1+\frac{n+1+\beta_1}{q_1}-
     \frac{n+1+\alpha_1}{p_1}, \\
     &-q_i a_i<\beta_i+1,\ \ \alpha_2=b_2,\ \ c_2<a_2+\frac{n+1+\beta_2}{q_2}.
     \end{aligned}
     \right.
     \end{align*}
\end{enumerate}
\end{theorem}

\begin{theorem}\label{0t4}
Let $\vp:=(1,1)$ and $\vq:=(q_1,q_2)\in[1,\fz)^2$.
Then the following statements are equivalent.
\begin{enumerate}
\item[{\rm(i)}]
The operator $S_{\negthinspace \va,\,\vb,\,\vc}$ is
bounded from $L^{\vp}_{\vec{\az}}$ to $\lt$\,.
\vspace{-0.2cm}
\item[{\rm(ii)}]
The operator $T_{\va,\,\vb,\,\vc}$ is bounded from $L^{\vp}_{\vec{\az}}$ to $\lt$\,.
\item[{\rm(iii)}]
The parameters satisfy that, for any $i\in\{1,2\}$,
\begin{align}\label{0e1}
     \left\{
     \begin{aligned}
     &-q_i a_i<\beta_i+1,\ \ \alpha_i<b_i,\\
     &c_i\le a_i+b_i-\alpha_i+\frac{n+1+\beta_i}{q_i},
     \end{aligned}
     \right.
\end{align}
or
\begin{align}\label{0e2}
     \left\{
     \begin{aligned}
     &-q_i a_i<\beta_i+1,\ \ \alpha_i=b_i,\\
     &c_i<a_i+\frac{n+1+\beta_i}{q_i},
     \end{aligned}
     \right.
\end{align}
or
\begin{align}\label{0e3}
     \left\{
     \begin{aligned}
     &-q_i a_i<\beta_i+1,\ \ \alpha_1=b_1,\ \ c_1< a_1+\frac{n+1+\beta_1}{q_1},\\
     &\alpha_2<b_2,\ \ c_2\le a_2+b_2-
     \alpha_2+\frac{n+1+\beta_2}{q_2},
     \end{aligned}
     \right.
\end{align}
or
\begin{align}\label{0e4}
     \left\{
     \begin{aligned}
     &-q_i a_i<\beta_i+1,\ \ \alpha_1<b_1,\ \ c_1\le a_1+b_1-
     \alpha_1+\frac{n+1+\beta_1}{q_1},\\
     &\alpha_2=b_2,\ \ c_2< a_2+\frac{n+1+\beta_2}{q_2}.
     \end{aligned}
     \right.
\end{align}
\end{enumerate}
\end{theorem}

The rest of this article is organized as follows.

In Section \ref{s2}, we present the proof of the necessary of boundedness for the operators $T_{\negthinspace 0,\,\vb,\,\vc}$. To achieve this, we first employ a crucial lemma (Lemma \ref{test2}) and then combine the boundedness of $T_{\negthinspace 0,\,\vb,\,\vc}$ with that of its adjoint operator to establish the necessity of Theorem \ref{0t1}. However, for the critical cases (Theorems \ref{0t2}, \ref{0t3}, and \ref{0t4}), this approach cannot be employed, as the function $f_{\xi,\zeta}$ in Lemma \ref{test2} does not belong to the space $\lv$ in these cases. To overcome this difficulty, we observe that although Lemma \ref{test2} is not available in these cases, the idea behind it is still effective. Specifically, we take full advantage of the boundedness of $T_{\negthinspace 0,\,\vb,\,\vc}$ and its adjoint operator and use the proof of Lemma \ref{test2} to obtain some necessary conditions on exponents in the considered critical cases, which are slightly larger than the desired ones. Moreover, by constructing an appropriate function $g_{\zeta}$ in Lemma \ref{2l8} and using a contradiction argument, we show that the necessary conditions on exponents are exactly the desired ones. Thus, by applying this approach, we finally prove the necessity of the boundedness of operators $T_{\negthinspace 0,\,\vb,\,\vc}$ in all the critical cases.

Section \ref{s1} focuses on establishing new Schur's tests in the setting of mixed-norm Lebesgue spaces and proving the sufficiency of boundedness for $S_{\negthinspace 0,\,\vb,\,\vc}$. Schur's test, which gives a sufficient condition for the boundedness of integral operators on $L^p$ spaces, is one of the most useful tools in establishing the boundedness of integral operators. Schur first introduced the criterion in \cite{s11} as a discrete form, and it has since been generalized by many authors. For instance, in 1965, Schur's test for the boundedness of one integral operator $T$ from $L^p$ to $L^q$ for $1<q\le p<\fz$ was obtained in \cite{g65}, and the work \cite{o70,z15} provided a generalization of Schur's test for the boundedness of $T$ from $L^p$ to $L^q$ for $1\le p\le q<\fz$. By constructing appropriate test functions and using these improved Schur's tests, we are able to establish the sufficiency of boundedness for $S_{\negthinspace 0,\,\vb,\,\vc}$.

In Section \ref{s3}, we conclude the proofs of our main results by combining the conclusions proved in Sections \ref{s2} and \ref{s1}, making two crucial observations, and putting the obtained pieces together. In the last section, we apply the main theorems to the weighted multiparameter Bergman projection and the weighted multiparameter Berezin transform. We obtain the necessary and sufficient conditions for the $L^{\vec p}-L^{\vec q}$ boundedness of the weighted multiparameter Berezin transform and the $L^{\vec p}-A^{\vec q}$ boundedness of the weighted multiparameter Bergman projection respectively, where $A^{\vec q}$ denotes the mixed-norm Bergman space.

Finally, we make some conventions on notation. We always denote by $C$
a \emph{positive constant} which is independent of the main parameters,
but it may vary from line to line. The notation $f\ls g$ means $f\le Cg$
and, if $f\ls g\ls f$, then we write $f\sim g$. If $f\le Cg$ and $g=h$
or $g\le h$, we then write $f\ls g\sim h$ or $f\ls g\ls h$, rather than
$f\ls g=h$ or $f\ls g\le h$. For any $q\in[1,\infty]$, we denote by $q'$ its
\emph{H\"older conjugate exponent}, namely, $1/q+1/q'=1$. 
Similarly, for any $\vq:=(q_1,q_2)\in[1,\infty]^2$, we denote by $\vq':=(q_1',q_2')$ its
\emph{H\"older conjugate exponent}, namely, $1/q_1+1/q_1'=1$ and $1/q_2+1/q_2'=1$.

\section{Proof of necessary for boundedness of $T_{\negthinspace 0,\,\vb,\,\vc}$
	\label{s2}}

In this section, we show the necessary of the main theorems.
For this purpose, we first show two useful lemmas as follows.

\begin{lemma}\label{anlyticint}
	Let $t>-1,\, c\in \mathbb{R}$, and $z \in \mathbb{B}_n$. For any $w \in \mathbb{B}_n$, define
	$f_z(w)=\frac{1}{(1-\langle z,w\rangle)^c}.$
	Then
	$\int_{\mathbb{B}_n}f_z(w)dv_t(w)=n\, \mathrm{B}(n,t+1),$
	where, for any $P,Q>0$, $\mathrm{B}(P,Q):=\int_0^1x^{P-1}(1-x)^{Q-1}\,dx$ denotes the Beta function.
\end{lemma}

\begin{proof}
	If $c=0$ or $z=0$, this lemma obviously holds true. Thus, we suppose $c\neq 0$ and $z\neq 0$. 
	Notice that, for fixed $z \in \mathbb{B}_n$, $1-\langle z,w \rangle$ is anti-analytic respect to $w$. Then $f_z$ is
	anti-analytic as well as harmonic on the closure of $\mathbb{B}_n$ if we extend its domain of definition. 
	Let $w=r\xi$ where $\xi \in \mathbb{S}_n$. By the mean value property, we have
	$1=f_z(0)=\int_{\mathbb{S}_n}f_z(r\xi)\,d\sigma(\xi)$.
	This further implies that
	$$\int_{\mathbb{B}_n}f_z(w)\,dv_t(w)=2n\int_{0}^{1}r^{2n-1}(1-r^2)^t\int_{\mathbb{S}_n}f_z(r\xi)
	\,d\sigma(\xi)\,dr=n\, \mathrm{B}(n,t+1),$$
	which completes the proof.
\end{proof}

\begin{lemma}\label{reproduce}
	Let $t>-1,\, c\in \mathbb{R}$, and $w \in \mathbb{B}_n$. For any $z \in \mathbb{B}_n$, let
	$g_w(z)=\frac{1}{(1-\langle z,w\rangle)^c}.$
	Then, for any $\xi\in \mathbb{B}_n$,
	$$g_w(\xi)=\int_{\mathbb{B}_n}\frac{g_w(u)\,dv_{t}(u)}{(1-\langle \xi,u \rangle)^{n+1+t}}.$$
\end{lemma}

\begin{proof}
	If $c=0$ or $w=0$, the result is obviously true by \cite[Theorem 2.2]{zbook05}.
	Therefore, we next suppose $c\neq 0$ and $w\neq 0$. By an argument similar
	to Lemma \ref{anlyticint}, we know that, for fixed $w\in \bn$, function $g_w(z)$ is analytic respect to $z$.
	Moreover, from the fact that $2\geq |1-\langle w,z \rangle|\geq 1-|\langle w,z \rangle|\geq 1-|w|,$
	we infer that, for fixed $w\in \bn$,
	$g_w\in L^\infty(\mathbb{B}_n,dv)\subset L^1(\mathbb{B}_n,dv_{t}).$
	Thus, by using \cite[Theorem 2.2]{zbook05}, we conclude that, for any $\xi\in \mathbb{B}_n$,
	$$g_w(\xi)=\int_{\mathbb{B}_n}\frac{g_w(u)dv_{t}(u)}{(1-\langle \xi,u \rangle)^{n+1+t}},$$
	which completes the proof.
\end{proof}

The following lemma is just \cite[Theorem 1.12]{zbook05} (see also \cite[Proposition 1.4.10]{rbook80}), 
which is needed when proving both the necessary and the sufficiency of main results.

\begin{lemma}\label{asymptotic}
	Suppose $z\in \mathbb{B}_n$, $t>-1$, and $c\in \rr$. The integral
	$$I_{c,t}(z)=\int_{\mathbb{B}_n}\frac{(1-|w|^2)^t}{|1-\langle z,w \rangle|^c}dv(w)$$
	has the following asymptotic behavior as $|z| \to 1^-$.
	\begin{enumerate}
		\item[{\rm(i)}] If $n+1+t-c>0$, then $I_{c,t}(z)\sim 1$.
		\item[{\rm(ii)}]  If $n+1+t-c=0$, then $I_{c,t}(z)\sim \log\frac{1}{1-|z|^2}$.
		\item[{\rm(iii)}]  If $n+1+t-c<0$, then $I_{c,t}(z)\sim (1-|z|^2)^{n+1+t-c}$.
	\end{enumerate}
\end{lemma}

As a consequence of Lemma \ref{asymptotic}, we have the following lemma which is useful 
both in the proofs of sufficiency and necessary.

\begin{lemma}\label{asymptotic2}
	Suppose $z\in \mathbb{B}_n$, $t>-1$, and $c\in \rr$. The integral
	$$I_{c,t}(z)=\int_{\mathbb{B}_n}\frac{(1-|w|^2)^t}{|1-\langle z,w \rangle|^c}dv(w)$$
	satisfies the following properties
	\begin{enumerate}
		\item[{\rm(i)}] If $n+1+t-c>0$, then $I_{c,t}(z)\lesssim 1$ for any $z\in \mathbb{B}_n$.
		\item[{\rm(ii)}]  If $n+1+t-c<0$, then $I_{c,t}(z)\lesssim (1-|z|^2)^{n+1+t-c}$ for any $z\in \mathbb{B}_n$.
	\end{enumerate}
\end{lemma}

\begin{proof}
	Let $n+1+t-c>0$. From Lemma \ref{asymptotic}(i), we deduce that there is a positive constant $C$
	and a constant $R\in(0,1)$ such that $I_{c,t}(z)\le C$ for any $R< |z|<1.$ If $|z|\leq R$, note that
	$1-R\le 1-|z||w|\le |1-\langle z,w \rangle|\leq 1+|z||w|\leq 2.$
	Therefore, for any $|z|\le R$, we have
	\begin{align}\label{1e0}
		I_{c,t}(z)\sim \int_{\mathbb{B}_n}(1-|w|^2)^tdv(w) \sim 1.
	\end{align}
	Thus, {\rm (i)} holds true.
	
	In addition, when $n+1+t-c<0$, by Lemma \ref{asymptotic}(iii), we find that there is a constant $R\in(0,1)$ such that $I_{c,t}(z)\ls (1-|z|^2)^{n+1+t-c}$ for any $R< |z|<1.$ On the other hand, for any $|z|\le R$,
	via \eqref{1e0} and the fact that $n+1+t-c<0$, we further conclude that $I_{c,t}(z)\sim 1\ls (1-|z|^2)^{n+1+t-c}$.
	Therefore, {\rm (ii)} is proved.
\end{proof}

To show the necessary for the boundedness of $T_{\negthinspace 0,\,\vb,\,\vc}$, we also need
the following result.

\begin{lemma}\label{boundeddual}
	Let $\vp:=(p_1,p_2)\in[1,\fz]^2$ and $\vq:=(q_1,q_2)\in[1,\fz]^2$. 
	If the integral operator $T_{0,\,\vb,\,\vc}$ is bounded from $\lv$ to $\lt$, 
	then its adjoint operator $T_{0,\,\vb,\,\vc}^*$ defined by setting
	\begin{align*}
		T_{0,\,\vb,\,\vc}^*g(z,w):=(1-|z|^2)^{b_1-\alpha_1}(1-|w|^2)^{b_2-\alpha_2}
		\int_{\mathbb{B}_n}\int_{\mathbb{B}_n}\frac{(1-|u|^2)^{\beta_1}(1-|\eta|^2)^{\beta_2}}{(1-\langle z,u \rangle)^{c_1}(1-\langle w,\eta \rangle)^{c_2}}g(u,\eta)\,dv(u)\,dv(\eta)
	\end{align*}
	is bounded from $L_{\vt}^{\vq'}$ to $L^{\vp'}_{\vf}$.
\end{lemma}
\begin{proof}
	Let $g \in L_{\vt}^{\vq'}$. From  \cite[p.\,304, Theorem 2]{bp61}, we infer that
	\begin{align*}
		||T_{0,\,\vb,\,\vc}^*g||_{L^{\vp'}_{\vf}}
		&=\sup\limits_{||f||_{L^{\vp}_{\vf}}=1}\left|\int_{\mathbb{B}_n}\int_{\mathbb{B}_n}\overline{f(z,w)}T_{0,\,\vb,\,\vc}^*g(z,w)\,dv_{\alpha_1}(z)\,dv_{\alpha_2}(w)\right|\\
		&=\sup\limits_{||f||_{L^{\vp}_{\vf}}=1}\left|\int_{\mathbb{B}_n}\int_{\mathbb{B}_n}\overline{f(z,w)}(1-|z|^2)^{b_1-\alpha_1}(1-|w|^2)^{b_2-\alpha_2}\right.\\
		&\hs\hs\left.\times\int_{\mathbb{B}_n}\int_{\mathbb{B}_n}\frac{(1-|u|^2)^{\beta_1}(1-|\eta|^2)^{\beta_2}}{(1-\langle z,u \rangle)^{c_1}(1-\langle w,\eta \rangle)^{c_2}}g(u,\eta)\,dv(u)\,dv(\eta)\,dv_{\alpha_1}(z)\,dv_{\alpha_2}(w)\right|.
	\end{align*}
	By the Fubini theorem, the H\"{o}lder inequality for mixed norms (see, for instance, \cite[Remark 2.8(iv)]{hlyy}), and the boundedness of $T_{0,\,\vb,\,\vc}$ from $\lv$ to $\lt$, we conclude that
	\begin{align*}
		||T_{0,\,\vb,\,\vc}^*g||_{L^{\vp'}_{\vf}}
		&=\sup\limits_{||f||_{L^{\vp}_{\vf}}=1}\left|\int_{\mathbb{B}_n}\int_{\mathbb{B}_n}g(u,\eta)\overline{T_{0,\,\vb,\,\vc}f(u,\eta)}\,dv_{\beta_1}(u)\,dv_{\beta_2}(\eta)\right|\\
		&\leq \sup\limits_{||f||_{L^{\vp}_{\vf}}=1}\lf\|T_{0,\,\vb,\,\vc}f\r\|_{\lt}|
		|g||_{L^{\vq'}_{\vt}}\ls ||g||_{L^{\vq'}_{\vt}}.
	\end{align*}
	This finishes the proof.
\end{proof}

The following two lemmas are important and used frequently in the proof of necessary.

\begin{lemma}\label{test1}
	Let $\vq:=(q_1,q_2)\in[1,\fz)^2$, $N_1$ and $N_2$ be two positive numbers such that $N_1+\beta_1>-1$ and $N_2+\beta_2>-1$. For any $z\in\bn$ and $w\in\bn$, define
	$f_{N_1,N_2}(z,w):=(1-|z|^2)^{N_1}(1-|w|^2)^{N_2}.$
	Then $f_{N_1,N_2} \in L^{\vq '}_{\vt}$ and there exist two positive constants $C_{(N_1,\,\beta_1)}$ and $C_{(N_2,\,\beta_2)}$ such that
	\begin{align}\label{2x1}
		T_{0,\,\vb,\,\vc}^*f_{N_1,N_2}(z,w)=C_{(N_1,\,\beta_1)}C_{(N_2,\,\beta_2)}
		(1-|z|^2)^{b_1-\alpha_1}(1-|w|^2)^{b_2-\alpha_2}.
	\end{align}
\end{lemma}
\begin{proof}
	It is easy to show that $f_{N_1,N_2} \in L^{\vq '}_{\vt}$.
	We next prove \eqref{2x1}. By the definition of the operator $T_{0,\,\vb,\,\vc}^*$, we know that
	\begin{align}\label{2e1}
		&T_{0,\,\vb,\,\vc}^*f_{N_1,N_2}(z,w)\notag\\
		&\hs=(1-|z|^2)^{b_1-\alpha_1}(1-|w|^2)^{b_2-\alpha_2}\int_{\mathbb{B}_n}\frac{(1-|u|^2)^{\beta_1+N_1}}{(1-\langle z,u \rangle)^{c_1}}\,dv(u)\int_{\mathbb{B}_n}\frac{(1-|\eta|^2)^{\beta_2+N_2}}{(1-\langle w,\eta \rangle)^{c_2}}\,dv(\eta)\notag\\
		&\hs=(1-|z|^2)^{b_1-\alpha_1}(1-|w|^2)^{b_2-\alpha_2}\int_{\mathbb{B}_n}\frac{1}{(1-\langle z,u \rangle)^{c_1}}\,dv_{\beta_1+N_1}(u)
		\int_{\mathbb{B}_n}\frac{1}{(1-\langle w,\eta \rangle)^{c_2}}\,dv_{\beta_2+N_2}(\eta).
	\end{align}
	Since $N_1+\beta_1>-1$, we apply Lemma \ref{anlyticint} with
	$f_z(u):=\frac{1}{(1-\langle z,u \rangle)^{c_1}}.$
	Thus, there exists a positive constant $C_{(N_1,\,\beta_1)}$, depending on $N_1$ and $\beta_1$, such that
	\begin{align}\label{2e2}
		\int_{\mathbb{B}_n}\frac{1}{(1-\langle z,u \rangle)^{c_1}}\,dv_{\beta_1+N_1}(u)=C_{(N_1,\,\beta_1)}.
	\end{align}
	Note that $N_2+\beta_2>-1$. Then, by using Lemma \ref{anlyticint} with
	$f_w(\eta):=\frac{1}{(1-\langle w,\eta \rangle)^{c_2}},$
	we conclude that there exists a positive constant $C_{(N_2,\,\beta_2)}$, depending on
	$N_2$ and $\beta_2$, such that
	$
	\int_{\mathbb{B}_n}\frac{1}{(1-\langle w,\eta \rangle)^{c_2}}\,dv_{\beta_2+N_2}(\eta)=C_{(N_2,\,\beta_2)}.
	$
	Combining this, \eqref{2e1}, and \eqref{2e2}, we further deduce that
	$$T_{0,\,\vb,\,\vc}^*f_{N_1,N_2}(z,w)=C_{(N_1,\,\beta_1)}C_{(N_2,\,\beta_2)}(1-|z|^2)^{b_1-\alpha_1}(1-|w|^2)^{b_2-\alpha_2}.$$
	Therefore, \eqref{2x1} holds true and hence we finish the proof of Lemma \ref{test1}.
\end{proof}

\begin{lemma}\label{test2}
	Let $\vp:=(p_1,p_2)\in[1,\fz)^2$, $\vq:=(q_1,q_2)\in[1,\fz)^2$, $\alpha_i+1\leq p_i(b_i+1)$ when $p_i>1$, and
	$\alpha_i<b_i$ when $p_i=1$ for $i\in\{1,2\}$.
	For any given $\xi,\zeta\in \mathbb{B}_n$, let
	\begin{align*}
		f_{\xi,\zeta}(z,w):=\frac{(1-|\xi|^2)^{n+1+b_1-(n+1+\alpha_1)/p_1}}{(1-\langle z,\xi \rangle)^{n+1+b_1}}\frac{(1-|\zeta|^2)^{n+1+b_2-(n+1+\alpha_2)/p_2}}{(1-\langle w,\zeta \rangle)^{n+1+b_2}}.
	\end{align*}
	Then $f_{\xi,\zeta}\in \lv$ and, if the operator $T_{0,\,\vb,\,\vc}$ is bounded from $\lv$ to $\lt$, then there exist two positive constants $C_1$ and $C_2$ such that, for any $\xi\in\bn$,
	\begin{align}\label{2e4}
		(1-|\xi|^2)^{(n+1+b_1)q_1-(n+1+\alpha_1)q_1/p_1}
		\int_{\mathbb{B}_n}\frac{(1-|z|^2)^{\beta_1}}
		{|1-\langle \xi,z \rangle|^{c_1q_1}}\,dv(z)\leq C_1
	\end{align}
	and, for any $\zeta\in\bn$,
	\begin{align}\label{2e5}
		(1-|\zeta|^2)^{(n+1+b_2)q_2-(n+1+\alpha_2)q_2/p_2}
		\int_{\mathbb{B}_n}\frac{(1-|w|^2)^{\beta_2}}
		{|1-\langle \zeta,w \rangle|^{c_2q_2}}\,dv(w)\leq C_2.
	\end{align}
\end{lemma}

\begin{proof}
	We first claim that $f_{\xi,\zeta}\in \lv$ and there exists a constant $M$ independent of $\xi$ and $\zeta$ such that $||f_{\xi,\zeta}||_{\lv}\leq M.$
	To achieve this, notice that
	\begin{align}\label{2e6}
		||f_{\xi,\zeta}||_{\lv}&=(1-|\xi|^2)^{n+1+b_1-(n+1+\alpha_1)/p_1}\left[ \int_{\mathbb{B}_n}\frac{(1-|z|^2)^{\alpha_1}}{|1-\langle z,\xi \rangle|^{(n+1+b_1)p_1}}\,dv(z)\right]^{1/p_1}\notag\\
		&\hs\hs\times (1-|\zeta|^2)^{n+1+b_2-(n+1+\alpha_2)/p_2}\left[ \int_{\mathbb{B}_n}\frac{(1-|w|^2)^{\alpha_2}}{|1-\langle w,\zeta \rangle|^{(n+1+b_2)p_2}}\,dv(w)\right]^{1/p_2}
	\end{align}
	and, for $i\in\{1,2\}$, $n+1+\alpha_i-(n+1+b_i)p_i=n(1-p_i)+[\alpha_i+1-p_i(b_i+1)]<0.$
	Thus, by this and Lemma \ref{asymptotic2}(ii), we conclude that, for any given $\xi \in\bn$,
	\begin{align*}
		\int_{\mathbb{B}_n}\frac{(1-|z|^2)^{\alpha_1}}{|1-\langle z,\xi \rangle|^{(n+1+b_1)p_1}}\,dv(z)\ls (1-|\xi|^2)^{n+1+\alpha_1-(n+1+b_1)p_1}
	\end{align*}
	and, for any given $\zeta \in\bn$,
	\begin{align*}
		\int_{\mathbb{B}_n}\frac{(1-|w|^2)^{\alpha_2}}{|1-\langle w,\zeta \rangle|^{(n+1+b_2)p_2}}\,dv(w)\ls (1-|\zeta|^2)^{n+1+\alpha_2-(n+1+b_2)p_2}.
	\end{align*}
	From these and \eqref{2e6}, we infer that $\|f_{\xi,\zeta}\|_{\lv}\ls 1$ and hence $f_{\xi,\zeta}\in \lv$.
	
	We next prove \eqref{2e4} and \eqref{2e5}. To this end,
	since $p_i(b_i+1)\geq\alpha_i+1>0$ for $i\in\{1,2\}$, we have $b_1>-1$ and $b_2>-1$.
	For any given $z \in \mathbb{B}_n$ and any $u\in\bn$, note that
	$2\ge|1-\langle u,z \rangle|\geq 1-|\langle u,z \rangle|\geq 1-|z|.$
	Thus,
	$f_z(u):=\frac{1}{(1-\langle u,z \rangle)^{c_1}}\in L^\infty(\mathbb{B}_n,dv)\subset L^1(\mathbb{B}_n,dv_{b_1}).$
	Similarly, for any given $w \in \mathbb{B}_n$, we have
	$g_w(\eta):=\frac{1}{(1-\langle \eta,w \rangle)^{c_2}}\in L^\infty(\mathbb{B}_n,dv)\subset L^1(\mathbb{B}_n,dv_{b_2}).$
	Then, applying Lemma \ref{reproduce} to $f_z(u)$ and $g_z(\eta)$, we deduce that
	\begin{align}\label{2e15}
		\frac{1}{(1-\langle \xi,z \rangle)^{c_1}}=\int_{\mathbb{B}_n}\frac{(1-|u|^2)^{b_1}\,dv(u)}{(1-\langle u,z \rangle)^{c_1}(1-\langle \xi,u \rangle)^{n+1+b_1}}
	\end{align}
	and
	\begin{align*}
		\frac{1}{(1-\langle \zeta,w \rangle)^{c_2}}=\int_{\mathbb{B}_n}\frac{(1-|\eta|^2)^{b_2}\,dv(\eta)}{(1-\langle \eta,w \rangle)^{c_2}(1-\langle \zeta,\eta \rangle)^{n+1+b_2}}.
	\end{align*}
	From this, \eqref{2e15}, and the definition of $T_{0,\,\vb,\,\vc}$, we deduce that
	\begin{align*}
		\overline{T_{0,\,\vb,\,\vc}f_{\xi,\zeta}(z,w)}=\frac{(1-|\xi|^2)^{(n+1+b_1)-(n+1+\alpha_1)/p_1}}
		{(1-\langle \xi,z \rangle)^{c_1}} \frac{(1-|\zeta|^2)^{(n+1+b_2)-(n+1+\alpha_2)/p_2}}
		{(1-\langle \zeta,w \rangle)^{c_2}}.
	\end{align*}
	Furthermore, since $T_{0,\,\vb,\,\vc}$ is bounded from $\lv$ to $\lt$ and by using the above assertion,
	we know that there exists a constant $M$, independent of $\xi$ and $\zeta$, such that
	\begin{align*}
		||\overline{T_{0,\,\vb,\,\vc}f_{\xi,\zeta}}||_{\lt}&=(1-|\xi|^2)^{(n+1+b_1)-(n+1+\alpha_1)/p_1}
		\left[\int_{\mathbb{B}_n}\frac{(1-|z|^2)^{\beta_1}}
		{|1-\langle \xi,z \rangle|^{c_1q_1}}\,dv(z)\right]^{1/q_1}\\
		&\hs\times (1-|\zeta|^2)^{(n+1+b_2)-(n+1+\alpha_2)/p_2}
		\left[\int_{\mathbb{B}_n}\frac{(1-|w|^2)^{\beta_2}}
		{|1-\langle \zeta,w \rangle|^{c_2q_2}}\,dv(w)\right]^{1/q_2}\\
		&\ls M.
	\end{align*}
	This implies \eqref{2e4} and \eqref{2e5} and hence finishes the proof.
\end{proof}

With the help of the above lemmas, we next prove the necessary of the main theorems as follows.

\begin{lemma}\label{2l1}
	Let $1<p_1\le q_1<\infty$ and $1<p_2\le q_2<\infty$.
	If the operator $T_{0,\,\vb,\,\vc}$ is bounded from $\lv$ to $\lt$, then
	$\alpha_1+1<p_1(b_1+1)$ and $\alpha_2+1<p_2(b_2+1)$.
\end{lemma}

\begin{proof}
	Let $N_1$ and $N_2$ be positive numbers such that $N_1+\beta_1>-1$ and $N_2+\beta_2>-1$. Define
	$f_{N_1,N_2}(z,w):=(1-|z|^2)^{N_1}(1-|w|^2)^{N_2}$
	for any $z\in\bn$ and $w\in\bn$.  By Lemma \ref{boundeddual}, we then know that $T_{0,\,\vb,\,\vc}^*$ is bounded from $L^{\vq'}_{\vt}$ to $L^{\vp'}_{\vf}$. From Lemma \ref{test1}, it follows that $T_{0,\,\vb,\,\vc}^*f_{N_1,N_2}\in L^{\vp'}_{\vf}$ and
	$T_{0,\,\vb,\,\vc}^*f_{N_1,N_2}(z,w)=C_{(N_1,\,\beta_1)}C_{(N_2,\,\beta_2)}
	(1-|z|^2)^{b_1-\alpha_1}(1-|w|^2)^{b_2-\alpha_2}.$
	Therefore
	\begin{align*}
		\lf\|T_{0,\,\vb,\,\vc}^*f_{N_1,N_2}\r\|_{L^{\vp'}_{\vf}}&=C_{(N_1,\,\beta_1)}
		C_{(N_2,\,\beta_2)}\left[\int_{\mathbb{B}_n}(1-|z|^2)^{p'_1(b_1-\alpha_1)+\alpha_1}\,dv(u)\right]^{1/p_1'}
		\\
		&\hs\hs\times \left[\int_{\mathbb{B}_n}(1-|w|^2)^{p'_2(b_2-\alpha_2)+\alpha_2}\,dv(\eta)\right]^{1/p_2'}\\
		&<\fz,
	\end{align*}
	which implies that $\alpha_1+1<p_1(b_1+1)$ and $\alpha_2+1<p_2(b_2+1)$.
\end{proof}

By Lemmas \ref{boundeddual} and \ref{test1} and also an argument similar to that used in the proof
of Lemma \ref{2l1}, we obatin the following three results; the details are omitted.

\begin{lemma}\label{2l2}
	Let $1<p_1\le q_1<\infty$ and $1=p_2\le q_2<\infty$.
	If the operator $T_{0,\,\vb,\,\vc}$ is bounded from $\lv$ to $\lt$, then
	$\alpha_1+1<p_1(b_1+1)$ and
	$\alpha_2\le b_2.$
\end{lemma}

\begin{lemma}\label{2l5}
	Let $1=p_1\le q_1<\infty$ and $1<p_2\leq q_2<\infty$.
	If the operator $T_{0,\,\vb,\,\vc}$ is bounded from $\lv$ to $\lt$, then
	$\alpha_1\leq b_1$ and
	$\alpha_2+1<p_2(b_2+1).$
\end{lemma}

\begin{lemma}\label{2l6}
	Let $1=p_1\le q_1<\infty$ and $1=p_2\le q_2<\infty$.
	If the operator $T_{0,\,\vb,\,\vc}$ is bounded from $\lv$ to $\lt$, then
	$\alpha_1\leq b_1$ and
	$\alpha_2\le b_2.$
\end{lemma}

\begin{lemma}\label{2l7}
	Let $1\le p_1\leq q_1<\infty$, $1\le p_2\leq q_2<\infty$, $\alpha_1+1<p_1(b_1+1)$, and
	$\alpha_2+1<p_2(b_2+1)$.
	If the operator $T_{0,\,\vb,\,\vc}$ is bounded from $\lv$ to $\lt$, then
	\begin{align*}
		\left\{
		\begin{aligned}
			&c_1\leq n+1+b_1+\frac{n+1+\beta_1}{q_1}-\frac{n+1+\alpha_1}{p_1},\\
			&c_2\leq n+1+b_2+\frac{n+1+\beta_2}{q_2}-\frac{n+1+\alpha_2}{p_2}.
		\end{aligned}
		\right.
	\end{align*}
\end{lemma}

\begin{proof}
	We first prove the inequality about $c_1$. For this purpose, observe that
	\begin{align}\label{2x2}
		(n+1+b_1)-(n+1+\alpha_1)/p_1&=[n(p_1-1)+p_1(b_1+1)-(\alpha_1+1)]/p_1>0.
	\end{align}
	If $n+1+\beta_1-c_1q_1>0$, or equivalently,
	$c_1<\frac{n+1+\beta_1}{q_1},$ then, using Lemma \ref{asymptotic2}(i) and \eqref{2x2}, we know that
	\begin{align*}
		&(1-|\xi|^2)^{(n+1+b_1)q_1-(n+1+\alpha_1)q_1/p_1}
		\int_{\mathbb{B}_n}\frac{(1-|z|^2)^{\beta_1}}
		{|1-\langle \xi,z \rangle|^{c_1q_1}}\,dv(z)\\
		&\hs\ls (1-|\xi|^2)^{(n+1+b_1)q_1-(n+1+\alpha_1)q_1/p_1}
		\leq C.
	\end{align*}
	If $n+1+\beta_1-c_1q_1=0$, or equivalently,
	$c_1=\frac{n+1+\beta_1}{q_1},$
	then, by Lemma \ref{asymptotic}(ii), \eqref{2x2}, and the fact that $x^{\vaz}\log x \to 0 $ as $x \to 0$ for any $\vaz>0$,
	we find that there exists a positive constant $R$ such that, for any $R<|\xi|<1$,
	\begin{align*}
		&(1-|\xi|^2)^{(n+1+b_1)q_1-(n+1+\alpha_1)q_1/p_1}
		\int_{\mathbb{B}_n}\frac{(1-|z|^2)^{\beta_1}}
		{|1-\langle \xi,z \rangle|^{c_1q_1}}\,dv(z)\\
		&\hs\ls -(1-|\xi|^2)^{(n+1+b_1)q_1-(n+1+\alpha_1)q_1/p_1}\log (1-|\xi|^2)
		\leq C.
	\end{align*}
	For $|\xi|\leq R$, it is obvious that
	$$(1-|\xi|^2)^{(n+1+b_1)q_1-(n+1+\alpha_1)q_1/p_1}
	\int_{\mathbb{B}_n}\frac{(1-|z|^2)^{\beta_1}}
	{|1-\langle \xi,z \rangle|^{c_1q_1}}\,dv(z)\leq C.$$
	If $n+1+\beta_1-c_1q_1<0$, or equivalently,
	$c_1>\frac{n+1+\beta_1}{q_1},$ then, applying Lemma \ref{asymptotic}(iii), 
	we find that there exists a constant $R$ such that, for $R<|\xi|<1$,
	\begin{align*}
		&(1-|\xi|^2)^{(n+1+b_1)q_1-(n+1+\alpha_1)q_1/p_1}
		\int_{\mathbb{B}_n}\frac{(1-|z|^2)^{\beta_1}}
		{|1-\langle \xi,z \rangle|^{c_1q_1}}\,dv(z)\\
		&\hs\sim (1-|\xi|^2)^{(n+1+b_1)q_1-(n+1+\alpha_1)q_1/p_1+n+1+\beta_1-c_1q_1},
	\end{align*}
	which, together with \eqref{2e4} of Lemma \ref{test2}, implies that
	$$
	\begin{cases}
		n+1+\beta_1-c_1q_1<0,\\
		(n+1+b_1)q_1-(n+1+\alpha_1)q_1/p_1+n+1+\beta_1-c_1q_1\geq 0.
	\end{cases}
	$$
	It is equivalent to
	$\frac{n+1+\beta_1}{q_1}<c_1\leq n+1+b_1+\frac{n+1+\beta_1}{q_1}-\frac{n+1+\alpha_1}{p_1}.$
	Therefore, by this and the above argument, we conclude that $$c_1\leq n+1+b_1+\frac{n+1+\beta_1}{q_1}-\frac{n+1+\alpha_1}{p_1}.$$
	
	We now prove the inequality about $c_2$. Observe that
	\begin{align}\label{2x3}
		(n+1+b_2)-(n+1+\alpha_2)/p_2&=[n(p_2-1)+p_2(b_2+1)-(\alpha_2+1)]/p_2>0.
	\end{align}
	When $n+1+\beta_2-c_2q_2>0$, or equivalently,
	$c_2<\frac{n+1+\beta_2}{q_2},$ from Lemma \ref{asymptotic2}(i) and \eqref{2x3}, we then infer that
	\begin{align*}
		&(1-|\zeta|^2)^{(n+1+b_2)q_2-(n+1+\alpha_2)q_2/p_2}
		\int_{\mathbb{B}_n}\frac{(1-|w|^2)^{\beta_2}}
		{|1-\langle \zeta,w \rangle|^{c_2q_2}}\,dv(w)\\
		&\hs\ls (1-|\zeta|^2)^{(n+1+b_2)q_2-(n+1+\alpha_2)q_2/p_2}\leq C.
	\end{align*}
	When $n+1+\beta_2-c_2q_2=0$, or equivalently,
	$c_2=\frac{n+1+\beta_2}{q_2},$
	then, by Lemma \ref{asymptotic}(ii) and \eqref{2x3} again, we know that
	there exist a positive constant $R$ such that, for any $R<|\zeta|<1$,
	\begin{align*}
		&(1-|\zeta|^2)^{(n+1+b_2)q_2-(n+1+\alpha_2)q_2/p_2}
		\int_{\mathbb{B}_n}\frac{(1-|w|^2)^{\beta_2}}
		{|1-\langle \zeta,w \rangle|^{c_2q_2}}\,dv(w)\\
		&\hs\ls -(1-|\zeta|^2)^{(n+1+b_2)q_2-(n+1+\alpha_2)q_2/p_2}\log (1-|\zeta|^2)\leq C.
	\end{align*}
	For the case $|\zeta|\leq R$, it holds true that
	$$(1-|\zeta|^2)^{(n+1+b_2)q_2-(n+1+\alpha_2)q_2/p_2}
	\int_{\mathbb{B}_n}\frac{(1-|w|^2)^{\beta_2}}
	{|1-\langle \zeta,w \rangle|^{c_2q_2}}\,dv(w)\leq C.$$
	When $n+1+\beta_2-c_2q_2<0$, or equivalently,
	$c_2>\frac{n+1+\beta_2}{q_2},$ then, by Lemma \ref{asymptotic}(iii), we
	obtain that there exists a constant $R$ such that, for $R<|\zeta|<1$,
	\begin{align*}
		&(1-|\zeta|^2)^{(n+1+b_2)q_2-(n+1+\alpha_2)q_2/p_2}
		\int_{\mathbb{B}_n}\frac{(1-|w|^2)^{\beta_2}}
		{|1-\langle \zeta,w \rangle|^{c_2q_2}}\,dv(w)\\
		&\hs\sim (1-|\zeta|^2)^{(n+1+b_2)q_2-(n+1+\alpha_2)q_2/p_2+n+1+\beta_2-c_2q_2},
	\end{align*}
	which, together with \eqref{2e5} of Lemma \ref{test2}, further implies that
	$$
	\begin{cases}
		n+1+\beta_2-c_2q_2<0,\\
		(n+1+b_2)q_2-(n+1+\alpha_2)q_2/p_2+n+1+\beta_2-c_2q_2\geq 0.
	\end{cases}
	$$
	It is equivalent to
	$\frac{n+1+\beta_2}{q_2}<c_2\leq n+1+b_2+\frac{n+1+\beta_2}{q_2}-\frac{n+1+\alpha_2}{p_2}.$
	From this and the above argument, we infer that
	$$c_2\leq n+1+b_2+\frac{n+1+\beta_2}{q_2}-\frac{n+1+\alpha_2}{p_2},$$
	which completes the whole proof.
\end{proof}

\begin{lemma}\label{2l8}
	Let $1\le p_1\leq q_1<\infty$, $1=p_2\leq q_2<\infty$, $\alpha_1+1<p_1(b_1+1)$, and $\alpha_2=b_2$.
	If the operator $T_{0,\,\vb,\,\vc}$ is bounded from $\lv$ to $\lt$, then
	\begin{align*}
		\left\{
		\begin{aligned}
			&c_1\leq n+1+b_1+\frac{n+1+\beta_1}{q_1}-\frac{n+1+\alpha_1}{p_1},\\
			&c_2< \frac{n+1+\beta_2}{q_2}.
		\end{aligned}
		\right.
	\end{align*}
\end{lemma}

\begin{proof}
	For any given $\xi,\,\zeta\in \mathbb{B}_n$, let
	\begin{align*}
		f_{\xi,\zeta}(z,w)&=\frac{(1-|\xi|^2)^{n+1+b_1-(n+1+\alpha_1)/p_1}}
		{(1-\langle z,\xi \rangle)^{n+1+b_1}}\frac{1}{(1-\langle w,\zeta \rangle)^{n+1+b_2}}.
	\end{align*}
	Notice that $n+1+\alpha_1-(n+1+b_1)p_1=n(1-p_1)+[\alpha_1+1-p_1(b_1+1)]<0$
	and $n+1+\alpha_2-(n+1+b_2)=0$.
	Then, by the proof of Lemma \ref{test2} and Lemma \ref{asymptotic}(ii), we find that as $|\zeta|\to 1^-$,
	\begin{align*}
		\|f_{\xi,\zeta}\|_{\lv}
		\sim \int_{\bn}\frac{(1-|w|^2)^{\alpha_2}}{|1-\langle w,\zeta \rangle|^{n+1+b_2}}\,dv(w)
		\sim \log \frac1{1-|\zeta|^2}
	\end{align*}
	and
	\begin{align*}
		\|T_{0,\,\vb,\,\vc}f_{\xi,\zeta}\|_{\lt}&=(1-|\xi|^2)^{(n+1+b_1)-(n+1+\alpha_1)/p_1}
		\left[\int_{\mathbb{B}_n}\frac{(1-|z|^2)^{\beta_1}}
		{|1-\langle \xi,z \rangle|^{c_1q_1}}\,dv(z)\right]^{1/q_1}\\
		&\hs\hs\times
		\left[\int_{\mathbb{B}_n}\frac{(1-|w|^2)^{\beta_2}}
		{|1-\langle \zeta,w \rangle|^{c_2q_2}}dv(w)\right]^{1/q_2}.
	\end{align*}
	From the assumption that $T_{0,\,\vb,\,\vc}$ is bounded from $\lv$ to $\lt$, we infer that
	there exists a positive constant $M$, independent of $\xi$ and $\zeta$, such that as $|\zeta|\to 1^-$,
	$\|T_{0,\,\vb,\,\vc} f_{\xi,\zeta}\|_{\lt}\le M\|f_{\xi,\zeta}\|_{\lv}$, namely,
	\begin{align*}
		&(1-|\xi|^2)^{(n+1+b_1)-(n+1+\alpha_1)/p_1}
		\left[\int_{\mathbb{B}_n}\frac{(1-|z|^2)^{\beta_1}}
		{|1-\langle \xi,z \rangle|^{c_1q_1}}\,dv(z)\right]^{1/q_1}
		\lf[\int_{\mathbb{B}_n}\frac{(1-|w|^2)^{\beta_2}}
		{|1-\langle \zeta,w \rangle|^{c_2q_2}}dv(w)\right]^{1/q_2}\\
		&\hs\ls M \log \frac1{1-|\zeta|^2}.
	\end{align*}
	By Lemma \ref{asymptotic}, we know that this happens if and only if $n+1+\beta_2-c_2q_2\ge 0$
	and there is a positive constant $C$, independent of $\xi$, such that
	\begin{align}\label{2eq1}
		(1-|\xi|^2)^{(n+1+b_1)-(n+1+\alpha_1)/p_1}
		\left[\int_{\mathbb{B}_n}\frac{(1-|z|^2)^{\beta_1}}
		{|1-\langle \xi,z \rangle|^{c_1q_1}}\,dv(z)\right]^{1/q_1}\le C.
	\end{align}
	We first prove the assertion about $c_1$ from \eqref{2eq1}. To do this, note that
	\begin{align}\label{2x4}
		(n+1+b_1)-(n+1+\alpha_1)/p_1&=[n(p_1-1)+p_1(b_1+1)-(\alpha_1+1)]/p_1>0.
	\end{align}
	If $n+1+\beta_1-c_1q_1>0$, or equivalently,
	$c_1<\frac{n+1+\beta_1}{q_1},$ then, by Lemma \ref{asymptotic2}(i) and \eqref{2x4}, we have
	\begin{align*}
		&(1-|\xi|^2)^{(n+1+b_1)q_1-(n+1+\alpha_1)q_1/p_1}
		\int_{\mathbb{B}_n}\frac{(1-|z|^2)^{\beta_1}}
		{|1-\langle \xi,z \rangle|^{c_1q_1}}\,dv(z)\\
		&\hs\ls (1-|\xi|^2)^{(n+1+b_1)q_1-(n+1+\alpha_1)q_1/p_1}
		\leq C.
	\end{align*}
	If $n+1+\beta_1-c_1q_1=0$, or equivalently,
	$c_1=\frac{n+1+\beta_1}{q_1},$
	then, from Lemma \ref{asymptotic}(ii), \eqref{2x4}, and the fact that
	$x^{\vaz}\log x \to 0 $ as $x \to 0$ for any $\vaz>0$ again, we deduce
	that there exists a positive constant $R$ such that for any $R<|\xi|<1$,
	\begin{align*}
		&(1-|\xi|^2)^{(n+1+b_1)q_1-(n+1+\alpha_1)q_1/p_1}
		\int_{\mathbb{B}_n}\frac{(1-|z|^2)^{\beta_1}}
		{|1-\langle \xi,z \rangle|^{c_1q_1}}\,dv(z)\\
		&\hs\ls -(1-|\xi|^2)^{(n+1+b_1)q_1-(n+1+\alpha_1)q_1/p_1}\log (1-|\xi|^2)
		\leq C.
	\end{align*}
	For $|\xi|\leq R$, we also have
	$$(1-|\xi|^2)^{(n+1+b_1)q_1-(n+1+\alpha_1)q_1/p_1}
	\int_{\mathbb{B}_n}\frac{(1-|z|^2)^{\beta_1}}
	{|1-\langle \xi,z \rangle|^{c_1q_1}}\,dv(z)\leq C.$$
	If $n+1+\beta_1-c_1q_1<0$, or equivalently,
	$c_1>\frac{n+1+\beta_1}{q_1},$ then, by Lemma \ref{asymptotic}(iii), 
	we know that there exists a constant $R$ such that, for any $R<|\xi|<1$,
	\begin{align*}
		&(1-|\xi|^2)^{(n+1+b_1)q_1-(n+1+\alpha_1)q_1/p_1}
		\int_{\mathbb{B}_n}\frac{(1-|z|^2)^{\beta_1}}
		{|1-\langle \xi,z \rangle|^{c_1q_1}}\,dv(z)\\
		&\hs\sim (1-|\xi|^2)^{(n+1+b_1)q_1-(n+1+\alpha_1)q_1/p_1+n+1+\beta_1-c_1q_1},
	\end{align*}
	which, combined with \eqref{2eq1}, implies that
	$$
	\begin{cases}
		n+1+\beta_1-c_1q_1<0,\\
		(n+1+b_1)q_1-(n+1+\alpha_1)q_1/p_1+n+1+\beta_1-c_1q_1\geq 0.
	\end{cases}
	$$
	It is equivalent to
	$\frac{n+1+\beta_1}{q_1}<c_1\leq n+1+b_1+\frac{n+1+\beta_1}{q_1}-\frac{n+1+\alpha_1}{p_1}.$
	This, together with the above argument, implies that $$c_1\leq n+1+b_1+\frac{n+1+\beta_1}{q_1}-\frac{n+1+\alpha_1}{p_1}.$$
	Thus, the assertion about $c_1$ is proved.
	
	We now show the inequality about $c_2$. To achieve this, we suppose, on the contrary, that $n+1+\beta_2-c_2q_2=0$
	and, for any given $\zeta\in \mathbb{B}_n$, let
	\begin{align*}
		g_{\zeta}(z,w)=\frac{(1-\langle \zeta,w \rangle)^{(n+1+\beta_2)/q_2}}
		{|1-\langle \zeta,w \rangle|^{n+1+\beta_2}}.
	\end{align*}
	Then, from Lemma \ref{anlyticint}, it follows that
	\begin{align*}
		T_{0,\,\vb,\,\vc}^*g_{\zeta}(z,w)
		&=(1-|z|^2)^{b_1-\az_1}\int_{\bn}\frac{(1-|u|^2)^{\bz_1}}
		{(1-\langle z,u \rangle)^{c_1}}\,dv(u)
		\int_{\mathbb{B}_n}\frac{(1-|\eta|^2)^{\beta_2}(1-\langle \zeta,\eta \rangle)^{(n+1+\beta_2)/q_2}}
		{(1-\langle w,\eta \rangle)^{c_2}|1-\langle \zeta,\eta \rangle|^{n+1+\beta_2}}\,dv(\eta)\\
		&\sim (1-|z|^2)^{b_1-\az_1}
		\int_{\mathbb{B}_n}\frac{(1-|\eta|^2)^{\beta_2}(1-\langle \zeta,\eta \rangle)^{(n+1+\beta_2)/q_2}}
		{(1-\langle w,\eta \rangle)^{c_2}|1-\langle \zeta,\eta \rangle|^{n+1+\beta_2}}\,dv(\eta).
	\end{align*}
	We next consider the following six cases.
	
	\emph{\textbf{Case 1)}} $1<p_1\le q_1<\fz$ and $1=p_2<q_2<\fz$. In this case, by Lemma \ref{asymptotic}(ii),
	we find that as $|\zeta|\to 1^-$,
	\begin{align*}
		\|g_{\zeta}\|_{L^{\vq'}_{\vec {\beta}}}
		\sim \lf[\int_{\bn}\frac{(1-|w|^2)^{\beta_2}}{|1-\langle \zeta,w \rangle|^{n+1+\beta_2}}\,dv(w)\r]^{1/q_2'}
		\sim \lf(\log \frac1{1-|\zeta|^2}\r)^{1/q_2'}
	\end{align*}
	and, applying the assumption $n+1+\beta_2-c_2q_2=0$ and Lemma \ref{asymptotic}(ii) again, we have
	\begin{align*}
		\|T_{0,\,\vb,\,\vc}^*g_{\zeta}\|_{L^{\vp'}_{\vec{\az}}}
		&\ge \|T_{0,\,\vb,\,\vc}^*g_{\zeta}\|_{L^{p_1'}_{\az_1}}(\zeta)\\
		&\sim \lf[\int_{\mathbb{B}_n}(1-|z|^2)^{(b_1-\az_1)p_1'+\az_1}\,dv(z)\r]^{1/p_1'}
		\int_{\mathbb{B}_n}\frac{(1-|\eta|^2)^{\beta_2}}
		{|1-\langle \zeta,\eta \rangle|^{n+1+\beta_2}}\,dv(\eta)
		\sim \log \frac1{1-|\zeta|^2}
	\end{align*}
	as $|\zeta|\to 1^-$. This, combined with Lemma \ref{boundeddual}, implies that
	the integral operator $T_{0,\,\vb,\,\vc}^*$ is bounded from $L_{\vt}^{\vq'}$ to $L^{\vp'}_{\vf}$.
	Namely, there exists a positive constant $L$, independent of $\zeta$, such that as $|\zeta|\to 1^-$,
	\begin{align*}
		\log \frac1{1-|\zeta|^2} \ls\|T_{0,\,\vb,\,\vc}^*g_{\zeta}\|_{L^{\vp'}_{\vec{\az}}}
		\le L\|g_{\zeta}\|_{L^{\vq'}_{\vec {\beta}}}
		\sim L \lf(\log \frac1{1-|\zeta|^2}\r)^{1/q_2'},
	\end{align*}
	which is impossible. Thus, $n+1+\beta_2-c_2q_2\neq 0$ and hence $c_2<\frac{n+1+\beta_2}{q_2}$ in this case.
	
	\emph{\textbf{Case 2)}} $1<p_1\le q_1<\fz$ and $1=p_2=q_2$. In this case, we have
	$\|g_{\zeta}\|_{L^{\vq'}_{\vec {\beta}}}\sim 1$
	and, from Case 1), it follows that as $|\zeta|\to 1^-$,
	$
	\|T_{0,\,\vb,\,\vc}^*g_{\zeta}\|_{L^{\vp'}_{\vec{\az}}}
	\gs \log \frac1{1-|\zeta|^2},
	$
	which, together with Lemma \ref{boundeddual}, implies that
	there exists a positive constant $L$, independent of $\zeta$, such that as $|\zeta|\to 1^-$,
	\begin{align*}
		\log \frac1{1-|\zeta|^2} \ls\|T_{0,\,\vb,\,\vc}^*g_{\zeta}\|_{L^{\vp'}_{\vec{\az}}}
		\le L\|g_{\zeta}\|_{L^{\vq'}_{\vec {\beta}}}
		\sim L,
	\end{align*}
	which is impossible. Thus, $c_2<\frac{n+1+\beta_2}{q_2}$ in this case.
	
	\emph{\textbf{Case 3)}} $1=p_1< q_1<\fz$ and $1=p_2<q_2<\fz$. In this case, by Lemma \ref{asymptotic}(ii),
	we find that as $|\zeta|\to 1^-$,
	\begin{align*}
		\|g_{\zeta}\|_{L^{\vq'}_{\vec {\beta}}}
		\sim \lf[\int_{\bn}\frac{(1-|w|^2)^{\beta_2}}{|1-\langle \zeta,w \rangle|^{n+1+\beta_2}}\,dv(w)\r]^{1/q_2'}
		\sim \lf(\log \frac1{1-|\zeta|^2}\r)^{1/q_2'}
	\end{align*}
	and, applying the assumption $n+1+\beta_2-c_2q_2=0$ and Lemma \ref{asymptotic}(ii) again, we have
	\begin{align*}
		\|T_{0,\,\vb,\,\vc}^*g_{\zeta}\|_{L^{\vp'}_{\vec{\az}}}
		\ge T_{0,\,\vb,\,\vc}^*g_{\zeta}(0,\zeta)
		\sim \int_{\mathbb{B}_n}\frac{(1-|\eta|^2)^{\beta_2}}
		{|1-\langle \zeta,\eta \rangle|^{n+1+\beta_2}}\,dv(\eta)\sim\log \frac1{1-|\zeta|^2}
	\end{align*}
	as $|\zeta|\to 1^-$. This, combined with Lemma \ref{boundeddual}, further implies that
	there exists a positive constant $L$, independent of $\zeta$, such that as $|\zeta|\to 1^-$,
	\begin{align*}
		\log \frac1{1-|\zeta|^2} \ls\|T_{0,\,\vb,\,\vc}^*g_{\zeta}\|_{L^{\vp'}_{\vec{\az}}}
		\le L\|g_{\zeta}\|_{L^{\vq'}_{\vec {\beta}}}
		\sim L \lf(\log \frac1{1-|\zeta|^2}\r)^{1/q_2'},
	\end{align*}
	which is impossible. Thus, $n+1+\beta_2-c_2q_2\neq 0$ and hence $c_2<\frac{n+1+\beta_2}{q_2}$ in this case.
	
	\emph{\textbf{Case 4)}} $1=p_1=q_1$ and $1=p_2<q_2<\fz$. In this case, by Lemma \ref{asymptotic}(ii),
	we find that as $|\zeta|\to 1^-$,
	\begin{align*}
		\|g_{\zeta}\|_{L^{\vq'}_{\vec {\beta}}}
		\sim \lf[\int_{\bn}\frac{(1-|w|^2)^{\beta_2}}{|1-\langle \zeta,w \rangle|^{n+1+\beta_2}}\,dv(w)\r]^{1/q_2'}
		\sim \lf(\log \frac1{1-|\zeta|^2}\r)^{1/q_2'}
	\end{align*}
	and, from Case 3), we infer that
	$
	\|T_{0,\,\vb,\,\vc}^*g_{\zeta}\|_{L^{\vp'}_{\vec{\az}}}\gs \log \frac1{1-|\zeta|^2}
	$
	as $|\zeta|\to 1^-$. This, together with Lemma \ref{boundeddual}, implies that
	there exists a positive constant $L$, independent of $\zeta$, such that as $|\zeta|\to 1^-$,
	\begin{align*}
		&\log \frac1{1-|\zeta|^2}
		\ls\|T_{0,\,\vb,\,\vc}^*g_{\zeta}\|_{L^{\vp'}_{\vec{\az}}}
		\le L\|g_{\zeta}\|_{L^{\vq'}_{\vec {\beta}}}
		\sim L \lf(\log \frac1{1-|\zeta|^2}\r)^{1/q_2'},
	\end{align*}
	which is impossible. Thus, $c_2<\frac{n+1+\beta_2}{q_2}$ in this case.
	
	\emph{\textbf{Case 5)}} $1=p_1<q_1<\fz$ and $1=p_2=q_2$. In this case, we have
	$\|g_{\zeta}\|_{L^{\vq'}_{\vec {\beta}}}\sim 1$
	and, by Case 3), we know that
	$
	\|T_{0,\,\vb,\,\vc}^*g_{\zeta}\|_{L^{\vp'}_{\vec{\az}}}
	\gs \log \frac1{1-|\zeta|^2}
	$
	as $|\zeta|\to 1^-$. From this and Lemma \ref{boundeddual}, it follows that
	there exists a positive constant $L$, independent of $\zeta$, such that as $|\zeta|\to 1^-$,
	\begin{align*}
		&\log \frac1{1-|\zeta|^2}
		\ls\|T_{0,\,\vb,\,\vc}^*g_{\zeta}\|_{L^{\vp'}_{\vec{\az}}}
		\le L\|g_{\zeta}\|_{L^{\vq'}_{\vec {\beta}}}
		\sim L,
	\end{align*}
	which is impossible. Thus, $c_2<\frac{n+1+\beta_2}{q_2}$ in this case.
	
	\emph{\textbf{Case 6)}} $1=p_1=q_1$ and $1=p_2=q_2$. In this case, note that
	$\|g_{\zeta}\|_{L^{\vq'}_{\vec {\beta}}}= 1$
	and $
	\|T_{0,\,\vb,\,\vc}^*g_{\zeta}\|_{L^{\vp'}_{\vec{\az}}}\gs \log \frac1{1-|\zeta|^2}
	$
	as $|\zeta|\to 1^-$. This, together with Lemma \ref{boundeddual}, implies that
	there exists a positive constant $L$, independent of $\zeta$, such that as $|\zeta|\to 1^-$,
	\begin{align*}
		&\log \frac1{1-|\zeta|^2}
		\ls\|T_{0,\,\vb,\,\vc}^*g_{\zeta}\|_{L^{\vp'}_{\vec{\az}}}
		\le L\|g_{\zeta}\|_{L^{\vq'}_{\vec {\beta}}}
		= L.
	\end{align*}
	This is impossible. Thus, $c_2<\frac{n+1+\beta_2}{q_2}$ in this case.
	Combining all these six cases, we complete the proof.
\end{proof}

\begin{lemma}\label{2l11}
	Let $1=p_1\leq q_1<\infty$, $1\le p_2\leq q_2<\infty$, $\alpha_1=b_1$, and $\alpha_2+1<p_2(b_2+1)$.
	If the operator $T_{0,\,\vb,\,\vc}$ is bounded from $\lv$ to $\lt$, then
	\begin{align*}
		\left\{
		\begin{aligned}
			&c_1< \frac{n+1+\beta_1}{q_1},\\
			&c_2\leq n+1+b_2+\frac{n+1+\beta_2}{q_2}-\frac{n+1+\alpha_2}{p_2}.
		\end{aligned}
		\right.
	\end{align*}
\end{lemma}

\begin{proof}
	This lemma is a symmetric case of Lemma \ref{2l8}. Therefore, 
	its proof is similar to that of Lemma \ref{2l8} with 
	$$ f_{\xi,\zeta}(z,w)=\frac1
	{(1-\langle z,\xi \rangle)^{n+1+b_1}}\frac{(1-|\zeta|^2)^{n+1+b_2-(n+1+\alpha_2)/p_2}}{(1-\langle w,\zeta \rangle)^{n+1+b_2}}$$
	for any given $\xi,\,\zeta\in \mathbb{B}_n$, and
	\begin{align*}
		g_{\xi}(z,w)=\frac{(1-\langle \xi,z \rangle)^{(n+1+\beta_1)/q_1}}
		{|1-\langle \xi,z \rangle|^{n+1+\beta_1}}
	\end{align*}
	for any given $\xi\in \mathbb{B}_n$.
	Thus, we omit the proof.
\end{proof}

\begin{lemma}\label{2l12}
	Let $1=p_1\leq q_1<\infty$, $1=p_2\leq q_2<\infty$, $\alpha_1=b_1$, and $\alpha_2=b_2$.
	If the operator $T_{0,\,\vb,\,\vc}$ is bounded from $\lv$ to $\lt$, then
	\begin{align*}
		\left\{
		\begin{aligned}
			&c_1< \frac{n+1+\beta_1}{q_1},\\
			&c_2< \frac{n+1+\beta_2}{q_2}.
		\end{aligned}
		\right.
	\end{align*}
\end{lemma}

\begin{proof}
	For any given $\xi,\,\zeta\in \mathbb{B}_n$, let
	\begin{align*}
		f_{\xi,\zeta}(z,w)=\frac1{(1-\langle z,\xi \rangle)^{n+1+b_1}}
		\frac1{(1-\langle w,\zeta \rangle)^{n+1+b_2}}.
	\end{align*}
	Notice that $n+1+\alpha_1-(n+1+b_1)=0$
	and $n+1+\alpha_2-(n+1+b_2)=0$.
	Then, by Lemma \ref{asymptotic}(ii), we conclude that as $|\xi|\to 1^-$ and $|\zeta|\to 1^-$,
	\begin{align*}
		\|f_{\xi,\zeta}\|_{\lv}
		\sim \int_{\bn}\frac{(1-|z|^2)^{\alpha_1}}{|1-\langle z,\xi \rangle|^{n+1+b_1}}\,dv(z)
		\int_{\bn}\frac{(1-|w|^2)^{\alpha_2}}{|1-\langle w,\zeta \rangle|^{n+1+b_2}}\,dv(w)
		\sim \log \frac1{1-|\xi|^2}\log \frac1{1-|\zeta|^2}
	\end{align*}
	and, via the proof of Lemma \ref{test2}, we find that
	\begin{align*}
		\|T_{0,\,\vb,\,\vc}f_{\xi,\zeta}\|_{\lt}=
		\left[\int_{\bn}\frac{(1-|z|^2)^{\beta_1}}
		{|1-\langle \xi,z \rangle|^{c_1q_1}}\,dv(z)\right]^{1/q_1}
		\left[\int_{\mathbb{B}_n}\frac{(1-|w|^2)^{\beta_2}}
		{|1-\langle \zeta,w \rangle|^{c_2q_2}}\,dv(w)\right]^{1/q_2}.
	\end{align*}
	Applying the assumption that $T_{0,\,\vb,\,\vc}$ is bounded from $\lv$ to $\lt$, we infer that
	there exists a positive constant $M$, independent of $\xi$ and $\zeta$, such that as $|\xi|\to 1^-$
	and $|\zeta|\to 1^-$,
	$\|T_{0,\,\vb,\,\vc} f_{\xi,\zeta}\|_{\lt}\le M\|f_{\xi,\zeta}\|_{\lv}$, namely,
	\begin{align*}
		\left[\int_{\mathbb{B}_n}\frac{(1-|z|^2)^{\beta_1}}
		{|1-\langle \xi,z \rangle|^{c_1q_1}}\,dv(z)\right]^{1/q_1}
		\lf[\int_{\mathbb{B}_n}\frac{(1-|w|^2)^{\beta_2}}
		{|1-\langle \zeta,w \rangle|^{c_2q_2}}\,dv(w)\right]^{1/q_2}
		\le M \log \frac1{1-|\xi|^2} \log \frac1{1-|\zeta|^2}.
	\end{align*}
	By Lemma \ref{asymptotic}, we know that this happens if and only if $n+1+\beta_1-c_1q_1\ge 0$
	and $n+1+\beta_2-c_2q_2\ge 0$.
	To complete the proof, it suffices to show $n+1+\beta_1-c_1q_1\neq 0$
	and $n+1+\beta_2-c_2q_2\neq 0$.
	
	Indeed, if we suppose $n+1+\beta_1-c_1q_1=0$, for any given $\xi\in \mathbb{B}_n$, let
	\begin{align*}
		g_{\xi}(z,w)=\frac{(1-\langle \xi,z \rangle)^{(n+1+\beta_1)/q_1}}
		{|1-\langle \xi,z \rangle|^{n+1+\beta_1}}.
	\end{align*}
	Then, from Lemma \ref{anlyticint}, it follows that
	\begin{align*}
		T_{0,\,\vb,\,\vc}^*g_{\xi}(z,w)
		=\int_{\bn}\frac{(1-|u|^2)^{\bz_1}(1-\langle \xi,u\rangle)^{(n+1+\beta_1)/q_1}}
		{(1-\langle z,u \rangle)^{c_1}|1-\langle \xi,u\rangle|^{n+1+\beta_1}}\,dv(u).
	\end{align*}
	We next consider the following four cases.
	
	\emph{\textbf{Case 1)}} $1=p_1< q_1<\fz$ and $1=p_2<q_2<\fz$. In this case, by Lemma \ref{asymptotic}(ii),
	we find that as $|\xi|\to 1^-$,
	\begin{align*}
		\|g_{\xi}\|_{L^{\vq'}_{\vec {\beta}}}
		\sim \lf[\int_{\bn}\frac{(1-|z|^2)^{\beta_1}}{|1-\langle \xi,z \rangle|^{n+1+\beta_1}}\,dv(z)\r]^{1/q_1'}
		\sim \lf(\log \frac1{1-|\xi|^2}\r)^{1/q_1'}
	\end{align*}
	and, applying the assumption $n+1+\beta_1-c_1q_1=0$ and Lemma \ref{asymptotic}(ii) again, we have
	\begin{align*}
		\|T_{0,\,\vb,\,\vc}^*g_{\xi}\|_{L^{\vp'}_{\vec{\az}}}
		\ge T_{0,\,\vb,\,\vc}^*g_{\xi}(\xi,0)
		\sim \int_{\mathbb{B}_n}\frac{(1-|u|^2)^{\beta_1}}
		{|1-\langle \xi,u \rangle|^{n+1+\beta_1}}\,dv(u)\sim\log \frac1{1-|\xi|^2}
	\end{align*}
	as $|\xi|\to 1^-$. This, combined with Lemma \ref{boundeddual}, further implies that
	there exists a positive constant $L$, independent of $\xi$, such that as $|\xi|\to 1^-$,
	\begin{align*}
		\log \frac1{1-|\xi|^2} \ls\|T_{0,\,\vb,\,\vc}^*g_{\xi}\|_{L^{\vp'}_{\vec{\az}}}
		\le L\|g_{\xi}\|_{L^{\vq'}_{\vec {\beta}}}
		\sim L \lf(\log \frac1{1-|\xi|^2}\r)^{1/q_1'},
	\end{align*}
	which is impossible. Thus, $n+1+\beta_1-c_1q_1\neq 0$ and hence $c_1<\frac{n+1+\beta_1}{q_1}$ in this case.
	
	\emph{\textbf{Case 2)}} $1=p_1<q_1<\fz$ and $1=p_2=q_2$. In this case, by Lemma \ref{asymptotic}(ii),
	we find that as $|\xi|\to 1^-$,
	\begin{align*}
		\|g_{\xi}\|_{L^{\vq'}_{\vec {\beta}}}
		=\lf[\int_{\bn}\frac{(1-|z|^2)^{\beta_1}}{|1-\langle \xi,z \rangle|^{n+1+\beta_1}}\,dv(z)\r]^{1/q_1'}
		\sim \lf(\log \frac1{1-|\xi|^2}\r)^{1/q_1'}
	\end{align*}
	and, from Case 1), we infer that $
	\|T_{0,\,\vb,\,\vc}^*g_{\xi}\|_{L^{\vp'}_{\vec{\az}}}\gs \log \frac1{1-|\xi|^2}
	$
	as $|\xi|\to 1^-$. This, together with Lemma \ref{boundeddual}, implies that
	there exists a positive constant $L$, independent of $\xi$, such that as $|\xi|\to 1^-$,
	\begin{align*}
		&\log \frac1{1-|\xi|^2}
		\ls\|T_{0,\,\vb,\,\vc}^*g_{\xi}\|_{L^{\vp'}_{\vec{\az}}}
		\le L\|g_{\xi}\|_{L^{\vq'}_{\vec {\beta}}}
		\sim L \lf(\log \frac1{1-|\xi|^2}\r)^{1/q_1'},
	\end{align*}
	which is impossible. Thus, $c_1<\frac{n+1+\beta_1}{q_1}$ in this case.
	
	\emph{\textbf{Case 3)}} $1=p_1=q_1$ and $1=p_2<q_2<\fz$. In this case, we have
	$\|g_{\xi}\|_{L^{\vq'}_{\vec {\beta}}}\sim 1$
	and, by Case 1), we know that
	$
	\|T_{0,\,\vb,\,\vc}^*g_{\xi}\|_{L^{\vp'}_{\vec{\az}}}
	\gs \log \frac1{1-|\xi|^2}
	$
	as $|\xi|\to 1^-$. From this and Lemma \ref{boundeddual}, it follows that
	there exists a positive constant $L$, independent of $\xi$, such that as $|\xi|\to 1^-$,
	$
	\log \frac1{1-|\xi|^2}
	\ls\|T_{0,\,\vb,\,\vc}^*g_{\xi}\|_{L^{\vp'}_{\vec{\az}}}
	\le L\|g_{\xi}\|_{L^{\vq'}_{\vec {\beta}}}
	\sim L,
	$
	which is impossible. Thus, $c_1<\frac{n+1+\beta_1}{q_1}$ in this case.
	
	\emph{\textbf{Case 4)}} $1=p_1=q_1$ and $1=p_2=q_2$. In this case, note that
	$\|g_{\xi}\|_{L^{\vq'}_{\vec {\beta}}}=1$
	and
	$
	\|T_{0,\,\vb,\,\vc}^*g_{\xi}\|_{L^{\vp'}_{\vec{\az}}}\gs \log \frac1{1-|\xi|^2}
	$
	as $|\xi|\to 1^-$. This, together with Lemma \ref{boundeddual}, implies that
	there exists a positive constant $L$, independent of $\xi$, such that as $|\xi|\to 1^-$,
	$
	\log \frac1{1-|\xi|^2}
	\ls\|T_{0,\,\vb,\,\vc}^*g_{\xi}\|_{L^{\vp'}_{\vec{\az}}}
	\le L\|g_{\xi}\|_{L^{\vq'}_{\vec {\beta}}}
	= L.$
	This is impossible. Thus, $c_1<\frac{n+1+\beta_1}{q_1}$ in this case.
	Combining these four cases, assertion about $c_1$ is proved.
	
	Similarly, for any given $\zeta\in \mathbb{B}_n$, let
	$
	g_{\zeta}(z,w)=\frac{(1-\langle \zeta,w \rangle)^{(n+1+\beta_2)/q_2}}
	{|1-\langle \zeta,w \rangle|^{n+1+\beta_2}}.
	$
	Then, by an argument similar to the proof of Lemma \ref{2l8},
	we also conclude that $c_2<\frac{n+1+\beta_2}{q_2}.$
	Thus, the proof is complete.
\end{proof}

\section{Proof of sufficiency for boundedness of $S_{\negthinspace 0,\,\vb,\,\vc}$ \label{s1}}

This section aims to show the sufficiency of the boundedness of
multiparameter Forelli-Rudin type operators $T_{\va,\,\vb,\,\vc}$ and
$S_{\negthinspace \va,\,\vb,\,\vc}$. To achieve this, we first need to extend
Schur's test to the weighted mixed-norm Lebesgue space as follows,
which generalizes the result of \cite[Theorem 1]{z15}.

\begin{proposition}\label{1p1}
Let $\vec{\mu}:=\mu_1\times\mu_2$ and $\vec{\nu}:=\nu_1\times\nu_2$ be positive measures on the
space $X\times X$ and, for $i\in\{1,2\}$, $K_i$ be nonnegative functions on $X\times X$.
Let $T$ be an integral operator with kernel $K:=K_1\cdot K_2$ defined
by setting for any $(x,y)\in X\times X$,
$$Tf(x,y):=\int_{X}\int_{X}K_1(x,s)K_2(y,t)f(s,t)\,d\mu_1(s)\,d\mu_2(t).$$
Suppose $\vp:=(p_1,p_2),\ \vq:=(q_1,q_2)\in(1,\fz)^2$ satisfying
$1<p_-\le p_+\le q_-<\fz$, where $p_+:=\max\{p_1,p_2\}$, $p_-:=\min\{p_1,p_2\}$,
and $q_-:=\min\{q_1,q_2\}$. Let $\gamma_i$ and $\delta_i$ be
real numbers such that $\gamma_i+\delta_i=1$ for $i\in\{1,2\}$.
If there exist two positive functions $h_1$ and $h_2$ defined
on $X\times X$ with two positive constants $C_1$ and $C_2$
such that for almost all $(x,y)\in X\times X$,
\begin{equation}\label{1e1}
\int_{X} \lf[\int_{X} [K_1(x,s)]^{\gamma_1 p_1'}[K_2(y,t)]^{\gamma_2 p_1'}[h_1(s,t)]^{p_1'}\,d\mu_1(s)\r]^{p_2'/p_1'}\,d\mu_2(t)
\le C_1[h_2(x,y)]^{p_2'}
\end{equation}
and for almost all $(s,t)\in X\times X$,
\begin{equation}\label{1e2}
\int_{X} \lf[\int_{X} [K_1(x,s)]^{\delta_1 q_1}[K_2(y,t)]^{\delta_2 q_1}[h_2(x,y)]^{q_1}\,d\nu_1(x)\r]^{q_2/q_1}\,d\nu_2(y)
\le C_2[h_1(s,t)]^{q_2},
\end{equation}
then
$T: L^{\vp}_{\vec{\mu}}\to L^{\vq}_{\vec{\nu}}$
is bounded with
$\|T\|_{L^{\vp}_{\vec{\mu}}\to L^{\vq}_{\vec{\nu}}}\le C_1^{1/p_2'}C_2^{1/q_2}.$
\end{proposition}

\begin{proof}
If $f\in L^{\vp}_{\vec{\mu}}$, then for almost every $(x,y)\in X\times X$, we have
$$|Tf(x,y)|\le \int_{X}\int_{X}K_1(x,s)K_2(y,t)[h_1(s,t)]^{-1}
h_1(s,t)|f(s,t)|\,d\mu_1(s)\,d\mu_2(t).$$
By using the H\"{o}lder inequality, we have
\begin{align*}
|Tf(x,y)|&\le\int_{X}\lf[\int_{X}[K_1(x,s)]^{\gamma_1 p_1'}[K_2(y,t)]^{\gamma_2 p_1'}
[h_1(s,t)]^{p_1'}\,d\mu_1(s)\r]^{1/p_1'}\\
&\hs\times\lf[\int_{X}[K_1(x,s)]^{\delta_1 p_1}[K_2(y,t)]^{\delta_2 p_1}
[h_1(s,t)]^{-p_1}|f(s,t)|^{p_1}\,d\mu_1(s)\r]^{1/p_1}\,d\mu_2(t)\\
&\le\lf\{\int_{X}\lf[\int_{X}[K_1(x,s)]^{\gamma_1 p_1'}[K_2(y,t)]^{\gamma_2 p_1'}
[h_1(s,t)]^{p_1'}\,d\mu_1(s)\r]^{p_2'/p_1'}\,d\mu_2(t)\r\}^{1/p_2'}\\
&\hs\times\lf\{\int_{X}\lf[\int_{X}[K_1(x,s)]^{\delta_1 p_1}[K_2(y,t)]^{\delta_2 p_1}
[h_1(s,t)]^{-p_1}|f(s,t)|^{p_1}\,d\mu_1(s)\r]^{p_2/p_1}\,d\mu_2(t)\r\}^{1/p_2},
\end{align*}
which, together with \eqref{1e1}, further implies that
\begin{align*}
&|Tf(x,y)|\\
&\hs\le C_1^{1/p_2'}h_2(x,y)\lf\{\int_{X}\lf[\int_{X}
[K_1(x,s)]^{\delta_1 p_1}[K_2(y,t)]^{\delta_2 p_1}[h_1(s,t)]^{-p_1}|f(s,t)|^{p_1}\,d\mu_1(s)\r]
^{p_2/p_1}\,d\mu_2(t)\r\}^{1/p_2}.
\end{align*}
In addition, from the assumption $1<p_-\le p_+\le q_-<\fz$,
it follows that $q_1\ge p_2$, $q_1\ge p_1$, $q_2\ge p_2$
and $q_2\ge p_1$. Thus, combining the above inequality and
the Minkowski inequality, we conclude that
\begin{align*}
&\|Tf\|_{L^{\vq}_{\vec{\nu}}}\\
&\hs\le C_1^{1/p_2'}\lf\{\int_{X}\lf[\int_{X}\lf(\int_{X}
\lf[\int_{X}[K_1(x,s)]^{\delta_1 q_1}[K_2(y,t)]^{\delta_2 q_1}
[h_2(x,y)]^{q_1}|f(s,t)|^{q_1}[h_1(s,t)]^{-q_1}\r.\r.\r.\r.\\
&\hs\hs\lf.\lf.\lf.\times\,d\nu_1(x)\Bigg]^{\frac{q_2}{q_1}}
\,d\nu_2(y)\r)^{\frac{p_1}{q_2}}\,d\mu_1(s)\r]^{\frac{p_2}{p_1}}\,d\mu_2(t)
\r\}^{\frac1{p_2}}\\
&\hs= C_1^{1/p_2'}\lf\{\int_{X}\lf[\int_{X}
|f(s,t)|^{p_1}[h_1(s,t)]^{-p_1}\lf(\int_{X}
\lf[\int_{X}[K_1(x,s)]^{\delta_1 q_1}[K_2(y,t)]^{\delta_2 q_1}
[h_2(x,y)]^{q_1}\r.\r.\r.\r.\\
&\hs\hs\lf.\lf.\lf.\times\,d\nu_1(x)\Bigg]^{\frac{q_2}{q_1}}
\,d\nu_2(y)\r)^{\frac{p_1}{q_2}}\,d\mu_1(s)\r]^{\frac{p_2}{p_1}}\,d\mu_2(t)
\r\}^{\frac1{p_2}}.
\end{align*}
Now applying \eqref{1e2}, we obtain
\begin{align*}
\|Tf\|_{L^{\vq}_{\vec{\nu}}}
&\le C_1^{1/p_2'}C_2^{1/q_2}\lf\{\int_{X}\lf[\int_{X}
|f(s,t)|^{p_1}\,d\mu_1(s)\r]^{\frac{p_2}{p_1}}\,d\mu_2(t)
\r\}^{\frac1{p_2}}
=C_1^{1/p_2'}C_2^{1/q_2} \|f\|_{L^{\vp}_{\vec{\mu}}}.
\end{align*}
This finishes the proof of Proposition \ref{1p1}.
\end{proof}

We also obtain the following three Schur's tests for the endpoint cases
$\vp:=(p_1,p_2)\equiv(1,1)=:\vec 1$, $1=p_1<p_2<\fz$ and $1=p_2<p_1<\fz$,
which are extensions of \cite[Theorem 2]{z15}.

\begin{proposition}\label{1p2}
Let $\vec{\mu}$, $\vec{\nu}$, the kernel $K$, and the operator $T$
be as in Proposition \ref{1p1}. Suppose $\vq:=(q_1,q_2)\in[1,\fz)^2$.
Let $\gamma_i$ and $\delta_i$ be two
real numbers such that $\gamma_i+\delta_i=1$ for $i\in\{1,2\}$.
If there exist two positive functions $h_1$ and $h_2$ defined
on $X\times X$ with two positive constants $C_1$ and $C_2$
such that for almost all $(x,y)\in X\times X$,
\begin{equation}\label{1e3}
\mathop{\esssup}\limits_{(s,t)\in X\times X}[K_1(x,s)]^{\gamma_1}[K_2(y,t)]^{\gamma_2}
h_1(s,t)\le C_1h_2(x,y)
\end{equation}
and, for almost all $(s,t)\in X\times X$,
\begin{equation}\label{1e4}
\int_{X} \lf[\int_{X} [K_1(x,s)]^{\delta_1 q_1}[K_2(y,t)]^{\delta_2 q_1}
[h_2(x,y)]^{q_1}\,d\nu_1(x)\r]^{q_2/q_1}\,d\nu_2(y)
\le C_2[h_1(s,t)]^{q_2},
\end{equation}
then
$T: L^{\vec 1}_{\vec{\mu}}\to L^{\vq}_{\vec{\nu}}$
is bounded with
$
\|T\|_{L^{\vec 1}_{\vec{\mu}}\to L^{\vq}_{\vec{\nu}}}\le C_1C_2^{1/q_2}.
$
\end{proposition}

\begin{proof}
Let $f\in L^{\vec 1}_{\vec{\mu}}$.
Then, by \eqref{1e3}, we find that, for almost every
$(x,y)\in X\times X$,
\begin{align*}
|Tf(x,y)|&\le \int_{X}\int_{X}K_1(x,s)K_2(y,t)[h_1(s,t)]^{-1}
h_1(s,t)|f(s,t)|\,d\mu_1(s)\,d\mu_2(t)\\
&\le \mathop{\esssup}\limits_{(s,t)\in X\times X}[K_1(x,s)]^{\gamma_1}[K_2(y,t)]^{\gamma_2}
h_1(s,t)\\
&\hs\times \int_{X}\int_{X}[K_1(x,s)]^{\delta_1}[K_2(y,t)]^{\delta_2}
[h_1(s,t)]^{-1}|f(s,t)|\,d\mu_1(s)\,d\mu_2(t)\\
&\le C_1h_2(x,y)\int_{X}\int_{X}[K_1(x,s)]^{\delta_1}[K_2(y,t)]^{\delta_2}
[h_1(s,t)]^{-1}|f(s,t)|\,d\mu_1(s)\,d\mu_2(t).
\end{align*}
Now we first show Proposition \ref{1p2} for the case
$\vq\in(1,\fz)^2$. In this case, applying the Minkowski
inequality, we deduce that
\begin{align*}
&\|Tf\|_{L^{\vq}_{\vec{\nu}}}\\
&\hs \le C_1\int_{X}\int_{X}\lf(\int_{X}\lf[\int_{X}
[K_1(x,s)]^{\delta_1 q_1}[K_2(y,t)]^{\delta_2 q_1}[h_2(x,y)]^{q_1}
[h_1(s,t)]^{-q_1}|f(s,t)|^{q_1}\r.\r.\\
&\hs\hs\lf.\times\,d\nu_1(x)\Bigg]^{\frac{q_2}{q_1}}\,d\nu_2(y)
\r)^{\frac1{q_2}}\,d\mu_1(s)\,d\mu_2(t)\\
&\hs= C_1\int_{X}\int_{X}|f(s,t)|[h_1(s,t)]^{-1}\\
&\hs\hs\times\lf(\int_{X}\lf[\int_{X}[K_1(x,s)]^{\delta_1 q_1}[K_2(y,t)]^{\delta_2 q_1}
[h_2(x,y)]^{q_1}\,d\nu_1(x)\r]^{\frac{q_2}{q_1}}\,d\nu_2(y)
\r)^{\frac1{q_2}}\,d\mu_1(s)\,d\mu_2(t),
\end{align*}
which, combined with \eqref{1e4}, further implies that
\begin{align*}
\|Tf\|_{L^{\vq}_{\vec{\nu}}}
&\le C_1C_2^{1/q_2}\int_{X}\int_{X}|f(s,t)|\,d\mu_1(s)\,d\mu_2(t)
= C_1C_2^{1/q_2}\|f\|_{L^{1}_{\vec{\mu}}}.
\end{align*}
This completes the proof in this case. For the case
$q_i=1$ for some $i\in\{1,2\}$, instead of using Minkowski's
inequality, applying the Fubini theorem will also
do the job. We omit the details in this case
and hence finish the proof.
\end{proof}

\begin{proposition}\label{1p3}
Let $\vec{\mu}$, $\vec{\nu}$, the kernel $K$, and the operator $T$
be as in Proposition \ref{1p1}. Suppose $\vp:=(1,p_2)$ with $p_2\in(1,\fz)$
and $\vq:=(q_1,q_2)\in(1,\fz)^2$ satisfying $1<p_2\le q_-<\fz$,
where $q_-:=\min\{q_1,q_2\}$. Let $\gamma_i$ and $\delta_i$ be two
real numbers such that $\gamma_i+\delta_i=1$ for $i\in\{1,2\}$.
If there exist two positive functions $h_1$ and $h_2$ defined
on $X\times X$ with two positive constants $C_1$ and $C_2$
such that for almost all $(x,y)\in X\times X$,
\begin{equation}\label{1e5}
\int_{X} \lf[\mathop{\esssup}\limits_{s\in X} [K_1(x,s)]^{\gamma_1}[K_2(y,t)]^{\gamma_2}h_1(s,t)\r]^{p_2'}\,d\mu_2(t)
\le C_1[h_2(x,y)]^{p_2'}
\end{equation}
and, for almost all $(s,t)\in X\times X$,
\begin{equation}\label{1e6}
\int_{X} \lf[\int_{X} [K_1(x,s)]^{\delta_1 q_1}[K_2(y,t)]^{\delta_2 q_1}[h_2(x,y)]^{q_1}\,d\nu_1(x)\r]^{q_2/q_1}\,d\nu_2(y)
\le C_2[h_1(s,t)]^{q_2},
\end{equation}
then
$T: L^{\vp}_{\vec{\mu}}\to L^{\vq}_{\vec{\nu}}$
is bounded with
$\|T\|_{L^{\vp}_{\vec{\mu}}\to L^{\vq}_{\vec{\nu}}}\le C_1^{1/p_2'}C_2^{1/q_2}.$
\end{proposition}

\begin{proof}
Let $f\in L^{\vp}_{\vec{\mu}}$. From \eqref{1e5} and the H\"{o}lder inequality,
we infer that, for almost every $(x,y)\in X\times X$,
\begin{align*}
|Tf(x,y)|&\le \int_{X}\int_{X}K_1(x,s)K_2(y,t)[h_1(s,t)]^{-1}
h_1(s,t)|f(s,t)|\,d\mu_1(s)\,d\mu_2(t)\\
&\le \int_{X} \mathop{\esssup}\limits_{s\in X} \lf\{[K_1(x,s)]^{\gamma_1}[K_2(y,t)]^{\gamma_2}h_1(s,t)\r\}\\
&\hs\times\int_{X}[K_1(x,s)]^{\delta_1}[K_2(y,t)]^{\delta_2}[h_1(s,t)]^{-1}|f(s,t)|\,d\mu_1(s)\,d\mu_2(t)\\
&\le \lf\{\int_{X} \lf[\mathop{\esssup}\limits_{s\in X} \{[K_1(x,s)]^{\gamma_1}[K_2(y,t)]^{\gamma_2}h_1(s,t)\}\r]^{p_2'}\,d\mu_2(t)\r\}^{1/p_2'}\\
&\hs\times\lf\{\int_{X}\lf[\int_{X}[K_1(x,s)]^{\delta_1}[K_2(y,t)]^{\delta_2 }
[h_1(s,t)]^{-1}|f(s,t)|\,d\mu_1(s)\r]^{p_2}\,d\mu_2(t)\r\}^{1/p_2}\\
&\le C_1^{1/p_2'}h_2(x,y)\lf\{\int_{X}\lf[\int_{X}
[K_1(x,s)]^{\delta_1}[K_2(y,t)]^{\delta_2 }[h_1(s,t)]^{-1}|f(s,t)|\,d\mu_1(s)\r]
^{p_2}\,d\mu_2(t)\r\}^{1/p_2}.
\end{align*}
Then, by the assumptions $q_1\ge p_2$ and $q_2\ge p_2$, and
the Minkowski inequality, we conclude that
\begin{align*}
&\|Tf\|_{L^{\vq}_{\vec{\nu}}}\\
&\hs\le C_1^{1/p_2'}\lf\{\int_{X}\lf[\int_{X}\lf(\int_{X}
\lf[\int_{X}[K_1(x,s)]^{\delta_1 q_1}[K_2(y,t)]^{\delta_2 q_1}
[h_2(x,y)]^{q_1}|f(s,t)|^{q_1}[h_1(s,t)]^{-q_1}\r.\r.\r.\r.\\
&\hs\hs\lf.\lf.\lf.\times\,d\nu_1(x)\Bigg]^{\frac{q_2}{q_1}}
\,d\nu_2(y)\r)^{\frac1{q_2}}\,d\mu_1(s)\r]^{p_2}\,d\mu_2(t)
\r\}^{\frac1{p_2}}\\
&\hs= C_1^{1/p_2'}\lf\{\int_{X}\lf[\int_{X}
|f(s,t)|[h_1(s,t)]^{-1}\lf(\int_{X}
\lf[\int_{X}[K_1(x,s)]^{\delta_1 q_1}[K_2(y,t)]^{\delta_2 q_1}
[h_2(x,y)]^{q_1}\r.\r.\r.\r.\\
&\hs\hs\lf.\lf.\lf.\times\,d\nu_1(x)\Bigg]^{\frac{q_2}{q_1}}
\,d\nu_2(y)\r)^{\frac1{q_2}}\,d\mu_1(s)\r]^{p_2}\,d\mu_2(t)
\r\}^{\frac1{p_2}},
\end{align*}
which, together with \eqref{1e6}, further implies that
\begin{align*}
\|Tf\|_{L^{\vq}_{\vec{\nu}}}
&\le C_1^{1/p_2'}C_2^{1/q_2}\lf\{\int_{X}\lf[\int_{X}
|f(s,t)|\,d\mu_1(s)\r]^{p_2}\,d\mu_2(t)
\r\}^{\frac1{p_2}}
=C_1^{1/p_2'}C_2^{1/q_2} \|f\|_{L^{\vp}_{\vec{\mu}}}.
\end{align*}
This finishes the proof of Proposition \ref{1p3}.
\end{proof}

\begin{proposition}\label{1p4}
Let $\vec{\mu}$, $\vec{\nu}$, the kernel $K$, and the operator $T$
be as in Proposition \ref{1p1}. Suppose $\vp:=(p_1,1)$ with $p_1\in(1,\fz)$
and $\vq:=(q_1,q_2)\in(1,\fz)^2$ satisfying
$1<p_1\le q_-<\fz$, where $q_-:=\min\{q_1,q_2\}$. Let $\gamma_i$ and
$\delta_i$ be two
real numbers such that $\gamma_i+\delta_i=1$ for $i\in\{1,2\}$.
If there exist two positive functions $h_1$ and $h_2$ defined
on $X\times X$ with two positive constants $C_1$ and $C_2$
such that for almost all $(x,y)\in X\times X$,
\begin{equation}\label{1e5'}
\mathop{\esssup}\limits_{t\in X} \int_{X} [K_1(x,s)]^{\gamma_1 p_1'}[K_2(y,t)]^{\gamma_2 p_1'}[h_1(s,t)]^{p_1'}\,d\mu_1(s)
\le C_1[h_2(x,y)]^{p_1'}
\end{equation}
and, for almost all $(s,t)\in X\times X$,
\begin{equation}\label{1e6'}
\int_{X} \lf[\int_{X} [K_1(x,s)]^{\delta_1 q_1}
[K_2(y,t)]^{\delta_2 q_1}[h_2(x,y)]^{q_1}\,d\nu_1(x)\r]^{q_2/q_1}\,d\nu_2(y)
\le C_2[h_1(s,t)]^{q_2},
\end{equation}
then
$T: L^{\vp}_{\vec{\mu}}\to L^{\vq}_{\vec{\nu}}$
is bounded with
$\|T\|_{L^{\vp}_{\vec{\mu}}\to L^{\vq}_{\vec{\nu}}}\le C_1^{1/p_1'}C_2^{1/q_2}.$
\end{proposition}

\begin{proof}
This proposition is a symmetric case of Proposition \ref{1p3}. Therefore, the proof is similar and hence
we omit it here.
\end{proof}

With the help of these Schur's tests, we next
prove the sufficiency of the main theorems.

\begin{lemma}\label{1l1}
     Let $1=p_+\leq q_-<\infty$ or $1<p_-\le p_+\leq q_-<\infty$.
     If the parameters satisfy for any $i\in\{1,2\}$,
     \begin{align*}
     \left\{
     \begin{aligned}
     &\alpha_i+1<p_i(b_i+1),\\
     &c_i\le n+1+b_i+\frac{n+1+\beta_i}{q_i}-
     \frac{n+1+\alpha_i}{p_i},
     \end{aligned}
     \right.
     \end{align*}
     then the operator $S_{\negthinspace 0,\,\vb,\,\vc}$ is
     bounded from $\lv$ to $\lt$.
\end{lemma}

\begin{proof}
For any $i\in\{1,2\}$, define
$$\lambda_i:=\frac{n+1+\beta_i}{q_i}-\frac{n+1+\alpha_i}{p_i},\ \ c_i:=n+1+b_i+\lambda_i,\ \
\text{and}\ \ \tau_i:=c_i-b_i+\alpha_i.$$ By the fact $-(1+\beta_i)/q_i<0$, we know that there exist two negative
numbers $r_1$ and $r_2$ such that, for any $i\in\{1,2\}$,
$
    -\frac{1+\beta_i}{q_i}<r_i<0.
$
In addition, from $\alpha_i+1<p_i(b_i+1)$, it follows that, for any $i\in\{1,2\}$,
\begin{align}\label{1e7}
b_i-\alpha_i+\frac{\alpha_i+1}{p'_i}>0.
\end{align}
Notice that, for any $i\in\{1,2\}$,
$
\tau_i =\frac{n+1+\alpha_i}{p'_i}+\frac{n+1+\beta_i}{q_i}>0,
$
which, combined with \eqref{1e7}, further implies that
\begin{align*}
    -\frac{\tau_i(1+\alpha_i)}{p'_i}-\frac{(b_i-\alpha_i)(n+1+\alpha_i)}{p'_i}
    <\frac{(b_i-\alpha_i)(n+1+\beta_i)}{q_i}.
\end{align*}
Then there exist $s_1$ and $s_2$ such that, for any $i\in\{1,2\}$,
    \begin{align*}
    -\frac{\tau_i(1+\alpha_i)}{p'_i}-\frac{(b_i-\alpha_i)(n+1+\alpha_i)}{p'_i}
    <\tau_is_i+(b_i-\alpha_i)(s_i-r_i)<\frac{(b_i-\alpha_i)(n+1+\beta_i)}{q_i},
  \end{align*}
which is equivalent to
    \begin{align}\label{1e9}
    -\frac{1+\alpha_i}{p'_i}-(b_i-\alpha_i)\gamma_i<s_i<(b_i-\alpha_i)\delta_i,
  \end{align}
where, for any $i\in\{1,2\}$,
$\gamma_i:=\frac{(n+1+\alpha_i)/p'_i+s_i-r_i}{\tau_i}$ 
and $\delta_i:=\frac{(n+1+\beta_i)/q_i+r_i-s_i}{\tau_i}.$
Clearly, $\gamma_i+\delta_i=1$. We now define
  $ h_1(u,\eta):=(1-|u|^2)^{s_1}(1-|\eta|^2)^{s_2},\ h_2(z,w):=(1-|z|^2)^{r_1}(1-|w|^2)^{r_2},$
  $$ K_1(z,u):=\frac{(1-|u|^2)^{b_1-\alpha_1}}{|1-\langle z,u \rangle|^{c_1}},\ \text{and}\ \ K_2(w,\eta)=\frac{(1-|\eta|^2)^{b_2-\alpha_2}}{|1-\langle w,\eta \rangle|^{c_2}}.$$
Therefore,
$$S_{\negthinspace 0,\,\vb,\,\vc}f(z,w)=\int_{\bn}\int_{\bn}K_1(z,u)K_2(w,\eta)f(u,\eta)
\,dv_{\az_1}(u)\,dv_{\az_2}(\eta).$$
We next prove this lemma by considering the following two cases.

\emph{\textbf{Case 1)}} $p_->1$. In this case, in order to use Proposition \ref{1p1} to
show the boundedness of $S_{\negthinspace 0,\,\vb,\,\vc}$, we now consider
\begin{align}\label{1e11}
&\int_{\mathbb{B}_n}[K_1(z,u)]^{\gamma_1 p'_1}[K_2(w,\eta)]^{\gamma_2 p'_1}[h_1(u,\eta)]^{p'_1}dv_{\alpha_1}(u)\noz\\
&\hs=\frac{(1-|\eta|^2)^{(b_2-\alpha_2)\gamma_2 p'_1+s_2p'_1}}
{|1-\langle w,\eta\rangle|^{c_2\gamma_2 p'_1}}\int_{\mathbb{B}_n}\frac{(1-|u|^2)^{(b_1-\alpha_1)\gamma_1 p'_1+s_1p'_1+\alpha_1}}{|1-\langle z,u\rangle|^{c_1\gamma_1 p'_1}}dv(u).
\end{align}
  From the left inequality of (\ref{1e9}), we have
  \begin{align}\label{1e10}
    (b_i-\alpha_i)\gamma_i p'_i+s_ip'_i+\alpha_i>-1.
  \end{align}
Moreover, by the fact that
  $(c_i-b_i+\alpha_i)\gamma_i=\tau_i\gamma_i=\frac{n+1+\alpha_i}{p'_i}+s_i-r_i,$
it follows that, for any $i\in\{1,2\}$,
  \begin{align}\label{1e12}
    n+1+(b_i-\alpha_i)\gamma_i p'_i+s_ip'_i+\alpha_i-c_i\gamma_i p'_i=r_ip'_i<0.
  \end{align}
From this, \eqref{1e10}, and Lemma \ref{asymptotic2}(ii), we infer that, for any given $z\in\bn$,
  \begin{align*}
    \int_{\mathbb{B}_n}\frac{(1-|u|^2)^{(b_1-\alpha_1)\gamma_1 p'_1+s_1p'_1+\alpha_1}}
    {|1-\langle z,u\rangle|^{c_1\gamma_1 p'_1}}
    dv(u)\ls (1-|z|^2)^{r_1p'_1},
  \end{align*}
which, together with \eqref{1e11}, \eqref{1e10}, \eqref{1e12}, and Lemma \ref{asymptotic2}(ii) again, further implies that
\begin{align*}
&\int_{\mathbb{B}_n}\left[\int_{\mathbb{B}_n}[K_1(z,u)]^{\gamma_1 p'_1}[K_2(w,\eta)]^{\gamma_2 p'_1}[h_1(u,\eta)]^{p'_1}dv_{\alpha_1}(u) \right]^{p'_2/p'_1}\,dv_{\alpha_2}(\eta)\\
&\hs\ls (1-|z|^2)^{r_1p'_2} \int_{\mathbb{B}_n}\frac{(1-|\eta|^2)^{(b_2-\alpha_2)\gamma_2 p'_2+s_2p'_2+\alpha_2}}
    {|1-\langle w,\eta\rangle|^{c_2\gamma_2 p'_2}}\,dv(\eta)\\
&\hs\ls (1-|z|^2)^{r_1p'_2} (1-|w|^2)^{r_2p'_2}
    \sim [h_2(z,w)]^{p'_2}.
  \end{align*}
Thus, condition \eqref{1e1} holds true for the operator $S_{\negthinspace 0,\,\vb,\,\vc}$.

We next check the second condition \eqref{1e2}. Observe that
  \begin{align*}
    &\int_{\mathbb{B}_n}[K_1(z,u)]^{\delta_1 q_1}[K_2(w,\eta)]^{\delta_2 q_1}[h_2(z,w)]^{q_1}\,dv_{\beta_1}(z)\\
    &\hs=\frac{(1-|\eta|^2)^{(b_2-\alpha_2)\delta_2 q_1}(1-|w|^2)^{r_2q_1}}{|1-\langle w,\eta \rangle|^{c_2\delta_2 q_1}}\int_{\mathbb{B}_n}\frac{(1-|u|^2)^{(b_1-\alpha_1)\delta_1 q_1}(1-|z|^2)^{r_1q_1+\beta_1}}{|1-\langle z,u \rangle|^{c_1\delta_1 q_1}}\,dv(z).
  \end{align*}
Via the definition of $r_i$, we obviously have
  \begin{align}\label{1e13}
    r_iq_i+\beta_i>-1.
  \end{align}
In addition, from the choice of $\delta_i$, it follows that, for any $i\in\{1,2\}$,
  $(c_i-b_i+\alpha_i)\delta_i=\tau_i\delta_i=\frac{n+1+\beta_i}{q_i}+r_i-s_i,$
which, combined with the right inequality of \eqref{1e9}, implies that
  \begin{align}\label{1e14}
    n+1+r_iq_i+\beta_i-c_i\delta q_i=s_iq_i-(b_i-\alpha_i)\delta_iq_i<0.
  \end{align}
From this, \eqref{1e13}, and Lemma \ref{asymptotic2}(ii), we deduce that
  \begin{align*}
    \int_{\mathbb{B}_n}\frac{(1-|z|^2)^{r_1q_1+\beta_1}}{|1-\langle z,u \rangle|^{c_1\delta q_1}}\,dv(z)
\ls (1-|u|^2)^{s_1q_1-(b_1-\alpha_1)\delta_1q_1}.
  \end{align*}
This, together with \eqref{1e13}, \eqref{1e14}, and Lemma \ref{asymptotic2}(ii)
again, further implies that
  \begin{align*}
  &\int_{\mathbb{B}_n}\left[\int_{\bn}[K_1(z,u)]^{\delta_1 q_1}[K_2(w,\eta)]^{\delta_2 q_1}[h_2(z,w)]^{q_1}\,dv_{\beta_1}(z) \right]^{q_2/q_1}\,dv_{\beta_2}(w)\\
  &\hs\ls (1-|u|^2)^{s_1q_2} (1-|\eta|^2)^{(b_2-\alpha_2)\delta_2 q_2} \int_{\mathbb{B}_n}\frac{(1-|w|^2)^{r_2q_2+\beta_2}}
    {|1-\langle w,\eta\rangle|^{c_2\delta_2 q_2}}\,dv(\eta)\\
    &\hs\ls (1-|u|^2)^{s_1q_2} (1-|\eta|^2)^{s_2q_2}
    \sim [h_1(u,\eta)]^{q_2}.
  \end{align*}
This is just inequality \eqref{1e2} for $S_{\negthinspace 0,\,\vb,\,\vc}$. Thus,
the operator $S_{\negthinspace 0,\,\vb,\,\vc}$ satisfies all conditions of Proposition \ref{1p1}.
Then, applying Proposition \ref{1p1}, we find that $S_{\negthinspace 0,\,\vb,\,\vc}$ is
bounded from $\lv$ to $\lt$ when $c_i= n+1+b_i+\lambda_i$ for $i\in\{1,2\}$ in the present case.

If $c_i< n+1+b_i+\lambda_i$ with $i\in\{1,2\}$, let $\sigma_i=n+1+b_i+\lambda_i-c_i$
and $\vec{\sigma}:=(\sigma_1,\sigma_2)$. Then, for any $i\in\{1,2\}$, $\sigma_i>0$
and from the above proof, it follows that $S_{\negthinspace 0,\,\vb,\,\vc+\vec{\sigma}}$ is
bounded from $\lv$ to $\lt$. Notice that
  $$\frac{1}{|1-\langle z,u\rangle|^{c_1}}=\frac{|1-\langle z,u\rangle|^{\sigma_1}}{|1-\langle z,u\rangle|^{n+1+b_1+\lambda_1}}\leq \frac{2^{\sigma_1}}{|1-\langle z,u\rangle|^{n+1+b_1+\lambda_1}}$$
and
    $$\frac{1}{|1-\langle w,\eta\rangle|^{c_2}}=\frac{|1-\langle w,\eta\rangle|^{\sigma_2}}{|1-\langle w,\eta\rangle|^{n+1+b_2+\lambda_2}}\leq \frac{2^{\sigma_2}}{|1-\langle w,\eta\rangle|^{n+1+b_2+\lambda_2}}.$$
Therefore, for any $f\in \lv$, we have
  $|S_{\negthinspace 0,\,\vb,\,\vc}f(z,w)|\leq 2^{\sigma_1+\sigma_2}
S_{\negthinspace 0,\,\vb,\,\vc+\vec{\sigma}}|f|(z,w),$
which implies that the operator $S_{\negthinspace 0,\,\vb,\,\vc}$ is also bounded from $\lv$ to $\lt$ when
$c_i< n+1+b_i+\lambda_i$ for $i\in\{1,2\}$, and hence completes the proof of Lemma \ref{1l1} in this case.

\emph{\textbf{Case 2)}} $p_+=1$. In this case, for any $i\in\{1,2\}$, note that
$\gamma_i:=\frac{s_i-r_i}{\tau_i}$, $\delta_i:=\frac{(n+1+\beta_i)/q_i+r_i-s_i}{\tau_i},$
and
\begin{align}\label{1e9'}
-(b_i-\alpha_i)\gamma_i<s_i<(b_i-\alpha_i)\delta_i.
\end{align}
In order to show the boundedness of $S_{\negthinspace 0,\,\vb,\,\vc}$ by
applying Proposition \ref{1p2}, we next consider
\begin{align*}
[K_1(z,u)]^{\gamma_1}[K_2(w,\eta)]^{\gamma_2}h_1(u,\eta)
    =\frac{(1-|u|^2)^{\gamma_1(b_1-\alpha_1)+s_1}
    (1-|\eta|^2)^{\gamma_2(b_2-\alpha_2)+s_2}}{|1-\langle z,u\rangle|^{c_1\gamma_1}|1-\langle w,\eta\rangle|^{c_2\gamma_2}}.
\end{align*}
From the Cauchy-Schwarz inequality, it follows that, for any $z\in\bn$ and $u\in\bn$,
\begin{align}\label{1e15}
    |1-\langle z,u \rangle|\geq 1-|\langle z,u \rangle|\geq 1-|z|=\frac{1-|z|^2}{1+|z|}\geq \frac{1-|z|^2}{2}
  \end{align}
and
    \begin{align}\label{1e16}
    |1-\langle z,u \rangle|\geq \frac{1-|u|^2}{2},
  \end{align}
  Since $\tau_i=c_i-b_i+\alpha_i$ and $\gamma_i =(s_i-r_i)/\tau_i$, we obtain
  \begin{align}\label{1e19}
  c_i\gamma_i=(b_i-\alpha_i)\gamma_i+s_i-r_i.
  \end{align}
By the left inequality of \eqref{1e9'}, we know that, for any $i\in\{1,2\}$,
$\gamma_i(b_i-\alpha_i)+s_i>0.$
This, combined with \eqref{1e15}, \eqref{1e16}, \eqref{1e19}, and the fact that $-r_i>0$, implies that,
for any given $z\in\bn$ and any $u\in\bn$,
\begin{align}\label{1e21}
\frac{(1-|u|^2)^{\gamma_1(b_1-\alpha_1)+s_1}(1-|z|^2)^{-r_1}}{|1-\langle z,u\rangle|^{c_1\gamma_1}}
=\lf(\frac{1-|u|^2)}{|1-\langle z,u\rangle|}\r)^{\gamma_1(b_1-\alpha_1)+s_1}
\lf(\frac{1-|z|^2}{|1-\langle z,u\rangle|}\r)^{-r_1}\leq C
  \end{align}
and, similarly, for any given $w\in\bn$ and any $\eta\in\bn$,
\begin{align*}
    \frac{(1-|\eta|^2)^{\gamma_2(b_2-\alpha_2)+s_2}(1-|w|^2)^{-r_2}}{|1-\langle w,\eta\rangle|^{c_2\gamma}}\leq C.
  \end{align*}
This, together with \eqref{1e21}, further conclude that, for any given $(z,w)\in\bn\times\bn$,
\begin{align*}
\mathop{\esssup}\limits_{(u,\eta)\in \mathbb{B}_n\times \mathbb{B}_n }[K_1(z,u)]^{\gamma_1}
[K_2(w,\eta)]^{\gamma_2}h_1(u,\eta)
    \leq C h_2(z,w).
\end{align*}
Thus, condition \eqref{1e3} holds true for the operator $S_{\negthinspace 0,\,\vb,\,\vc}$.
Note that, in this case,  \eqref{1e4} is just \eqref{1e2} since no $p_i$ appears in this inequality.
Therefore, from the proof of \eqref{1e2} in Case 1), we infer that the operator $S_{\negthinspace 0,\,\vb,\,\vc}$
also satisfies \eqref{1e4} in the present case, and hence
all conditions of Proposition \ref{1p2} is satisfied for $S_{\negthinspace 0,\,\vb,\,\vc}$.

Thus, applying Proposition \ref{1p2}, we find that $S_{\negthinspace 0,\,\vb,\,\vc}$ is
bounded from $\lv$ to $\lt$ when $c_i= n+1+b_i+\lambda_i$ for $i\in\{1,2\}$. In addition,
by an argument similar to that used in the proof of Case1),
we find that the operator $S_{\negthinspace 0,\,\vb,\,\vc}$ is also bounded from $\lv$ to $\lt$ when
$c_i< n+1+b_i+\lambda_i$ for $i\in\{1,2\}$, and hence completes the whole proof of Lemma \ref{1l1}.
\end{proof}

\begin{lemma}\label{1l3}
 Let $1=p_2<p_1\leq q_-<\infty$.
     If the parameters satisfy
     \begin{align*}
     \left\{
     \begin{aligned}
     &\alpha_1+1<p_1(b_1+1),\ \ c_1\le n+1+b_1+\frac{n+1+\beta_1}{q_1}-
     \frac{n+1+\alpha_1}{p_1}, \\
     &\alpha_2=b_2,\ \ c_2<\frac{n+1+\beta_2}{q_2},
     \end{aligned}
     \right.
     \end{align*}
     then the operator $S_{\negthinspace 0,\,\vb,\,\vc}$ is
     bounded from $\lv$ to $\lt$.
\end{lemma}

\begin{proof}
Suppose $$\lambda_1:=\frac{n+1+\beta_1}{q_1}-\frac{n+1+\alpha_1}{p_1},\ \ c_1:=n+1+b_1+\lambda_1,\ \
\tau_1:=c_1-b_1+\alpha_1,$$ and $c_2\neq 0$. By the fact $-(1+\beta_i)/q_i<0$ with $i\in \{1,2\}$,
we know that there exist two negative
numbers $r_1$ and $r_2$ such that, for any $i \in \{1,2\}$,
$
    -\frac{1+\beta_i}{q_i}<r_i<0.
$
In addition, for any $i\in\{1,2\}$, we have
$
\tau_1=\frac{n+1+\alpha_1}{p'_1}+\frac{n+1+\beta_1}{q_1}>0.
$
This, combined with the fact that $b_1-\alpha_1+\frac{\alpha_1+1}{p'_1}>0,$
further implies that,
  \begin{align*}
    -\frac{\tau_1(1+\alpha_1)}{p'_1}-\frac{(b_1-\alpha_1)(n+1+\alpha_1)}{p'_1}
    <\frac{(b_1-\alpha_1)(n+1+\beta_1)}{q_1}.
  \end{align*}
Thus, there exists $s_1$ such that
    \begin{align*}
    -\frac{\tau_1(1+\alpha_1)}{p'_1}-\frac{(b_1-\alpha_1)(n+1+\alpha_1)}{p'_1}
    <\tau_1s_1+(b_1-\alpha_1)(s_1-r_1)<\frac{(b_1-\alpha_1)(n+1+\beta_1)}{q_1}.
  \end{align*}
That is
    \begin{align}\label{1e24'}
    -\frac{1+\alpha_1}{p'_1}-(b_1-\alpha_1)\gamma_1<s_1<(b_1-\alpha_1)\delta_1,
  \end{align}
where $\gamma_1:=\frac{(n+1+\alpha_1)/p'_1+s_1-r_1}{\tau_1}$
and $\delta_1:=\frac{(n+1+\beta_1)/q_1+r_1-s_1}{\tau_1}.$
Clearly, $\gamma_1+\delta_1=1$. We now define
  $\gamma_2:=-\frac{r_2}{c_2}$ and $\delta_2:=1+\frac{r_2}{c_2}.$
  Obviously, $\gamma_2+\delta_2=1$.
Let
  $ h_1(u,\eta):=(1-|u|^2)^{s_1}$, $h_2(z,w):=(1-|z|^2)^{r_1}(1-|w|^2)^{r_2},$
  $$ K_1(z,u):=\frac{(1-|u|^2)^{b_1-\alpha_1}}{|1-\langle z,u \rangle|^{c_1}},\ \text{and}\ \
K_2(w,\eta)=\frac1{|1-\langle w,\eta \rangle|^{c_2}}.$$
Therefore,
$$S_{\negthinspace 0,\,\vb,\,\vc}f(z,w)=\int_{\bn}\int_{\bn}K_1(z,u)K_2(w,\eta)f(u,\eta)
\,dv_{\az_1}(u)\,dv_{\az_2}(\eta).$$
In order to use Proposition \ref{1p4} to show the boundedness of
$S_{\negthinspace 0,\,\vb,\,\vc}$, we next consider
\begin{align}\label{1e25'}
&\int_{\mathbb{B}_n}[K_1(z,u)]^{\gamma_1 p'_1}[K_2(w,\eta)]^{\gamma_2 p'_1}[h_1(u,\eta)]^{p'_1}dv_{\alpha_1}(u)\noz\\
&\hs=\frac{1}{|1-\langle w,\eta\rangle|^{c_2\gamma_2 p'_1}}
\int_{\bn}\frac{(1-|u|^2)^{(b_1-\alpha_1)\gamma_1 p'_1+s_1p_1'+\alpha_1}}{|1-\langle z,u\rangle|^{c_1\gamma_1 p'_1}}dv(u).
\end{align}
Indeed, from \eqref{1e15} and the fact $c_2\gamma_2=-r_2>0$, we infer that, for any $w\in\bn$ and $\eta\in\bn$,
  \begin{align}\label{1e26'}
    \frac{(1-|w|^2)^{-r_2}}{|1-\langle w,\eta\rangle|^{c_2\gamma_2}}\leq C.
  \end{align}
By the left inequality of \eqref{1e24'}, we know that
  \begin{align}\label{1e27'}
   (b_1-\alpha_1)\gamma_1p'_1+s_1p'_1+\alpha_1>-1.
  \end{align}
Moreover, from the fact that
  $(c_1-b_1+\alpha_1)\gamma_1=\tau_1\gamma_1=\frac{n+1+\alpha_1}{p'_1}+s_1-r_1,$
it follows that $
     n+1+(b_1-\alpha_1)\gamma_1 p'_1+s_1p'_1+\alpha_1-c_1\gamma_1 p'_1=r_1p'_1<0.$
By this, \eqref{1e27'}, and Lemma \ref{asymptotic2}(ii), we conclude that, for any given $z\in\bn$,
  \begin{align*}
    \int_{\bn}\frac{(1-|u|^2)^{(b_1-\alpha_1)\gamma_1 p'_1+s_1p_1'+\alpha_1}}{|1-\langle z,u\rangle|^{c_1\gamma_1 p'_1}}dv(u)\ls (1-|z|^2)^{r_1p'_1},
  \end{align*}
which, together with \eqref{1e25'} and \eqref{1e26'}, further implies that
\begin{align*}
&\mathop{\esssup}\limits_{\eta\in \mathbb{B}_n}\int_{\mathbb{B}_n}[K_1(z,u)]^{\gamma_1 p'_1}
[K_2(w,\eta)]^{\gamma_2 p'_1}[h_1(u,\eta)]^{p'_1}dv_{\alpha_1}(u)\\
&\hs\ls (1-|z|^2)^{r_1p'_1}(1-|w|^2)^{r_2p'_1}
\ls [h_2(z,w)]^{p'_1}.
  \end{align*}
Thus, condition \eqref{1e5'} holds true for the operator $S_{\negthinspace 0,\,\vb,\,\vc}$.

We next check the second condition \eqref{1e6'} of Proposition \ref{1p4}. Observe that
  \begin{align}\label{1e1''}
    &\int_{\mathbb{B}_n}[K_1(z,u)]^{\delta_1 q_1}[K_2(w,\eta)]^{\delta_2 q_1}[h_2(z,w)]^{q_1}\,dv_{\beta_1}(z)\noz\\
    &\hs=\frac{(1-|w|^2)^{r_2q_1}(1-|u|^2)^{(b_1-\alpha_1)\delta_1 q_1}}{|1-\langle w,\eta \rangle|^{c_2\delta_2 q_1}}\int_{\mathbb{B}_n}\frac{(1-|z|^2)^{r_1q_1+\beta_1}}{|1-\langle z,u \rangle|^{c_1\delta_1 q_1}}\,dv(z).
  \end{align}
By the definition of $r_i$, we obviously have
  \begin{align}\label{1e28'}
    r_iq_i+\beta_i>-1
  \end{align}
for any $i \in \{1,2\}$. Since $\delta_2=1+r_2/c_2$ and $c_2<(n+1+\beta_2)/q_2$, it follows that
  \begin{align*}
    n+1+r_2q_2+\beta_2-c_2\delta_2q_2=q_2\lf(\frac{n+1+\beta_2}{q_2}-c_2\r)>0.
  \end{align*}
From this, \eqref{1e28'}, and Lemma \ref{asymptotic2}(i),  we deduce that, for any given $u\in\bn$,
  \begin{align}\label{1e29'}
    \int_{\mathbb{B}_n}\frac{(1-|w|^2)^{r_2q_2+\beta_2}}{|1-\langle w,\eta \rangle|^{c_2\delta_2 q_2}}dv(w)\leq C.
  \end{align}
Moreover, from the choice of $\delta_1$, we know that
  $(c_1-b_1+\alpha_1)\delta_1=\tau_1\delta_1=\frac{n+1+\beta_1}{q_1}+r_1-s_1,$
which, combined with the right inequality of \eqref{1e24'}, implies that
$n+1+r_1q_1+\beta_1-c_1\delta_1 q_1=s_1q_1-(b_1-\alpha_1)\delta_1q_1<0.$
Applying this, \eqref{1e28'}, and Lemma \ref{asymptotic2}(ii), we know that, for any given $\eta\in\bn$,
  \begin{align*}
    \int_{\mathbb{B}_n}\frac{(1-|z|^2)^{r_1q_1+\beta_1}}{|1-\langle z,u \rangle|^{c_1\delta_1 q_1}}\,dv(z)
\ls (1-|u|^2)^{s_1q_1-(b_1-\alpha_1)\delta_1q_1}.
  \end{align*}
This, together with \eqref{1e1''} and \eqref{1e29'}, further implies that
  \begin{align*}
  &\int_{\mathbb{B}_n}\left[\int_{\bn}[K_1(z,u)]^{\delta_1 q_1}[K_2(w,\eta)]^{\delta_2 q_1}[h_2(z,w)]^{q_1}\,dv_{\beta_1}(z) \right]^{q_2/q_1}\,dv_{\beta_2}(w)\\
  &\hs\ls (1-|u|^2)^{s_1q_2} \int_{\mathbb{B}_n}\frac{(1-|w|^2)^{r_2q_2+\beta_2}}
    {|1-\langle w,\eta\rangle|^{c_2\delta_2 q_2}}\,dv(\eta)
    \ls [h_1(u,\eta)]^{q_2}.
  \end{align*}
This is just inequality \eqref{1e6'} for $S_{\negthinspace 0,\,\vb,\,\vc}$. Thus,
the operator $S_{\negthinspace 0,\,\vb,\,\vc}$ satisfies all conditions of Proposition \ref{1p4}.
Then, applying Proposition \ref{1p4}, we find that $S_{\negthinspace 0,\,\vb,\,\vc}$ is
bounded from $\lv$ to $\lt$ in the case $c_1= n+1+b_2+\lambda_1$ and $c_2\neq0$.
When $c_1= n+1+b_2+\lambda_1$ and $c_2=0$, observe that there exists a positive
number $c_0$ such that $0<c_0<(n+1+\beta_2)/q_2$. Clearly, $|1-\langle w,\eta\rangle|\ls 1$
for any $w\in\bn$ and $\eta\in\bn$.
Thus, $1\ls \frac 1{|1-\langle w,\eta\rangle|^{c_0}}$ and hence
$|S_{\negthinspace 0,\,\vb,\,(c_1,0)}f(z,w)|\ls S_{\negthinspace 0,\,\vb,\,(c_1,c_0)}|f|(z,w).$
Therefore, by the above proof, we conclude that the operator $S_{\negthinspace 0,\,\vb,\,(c_1,0)}$
is also bounded from $\lv$ to $\lt$ in the case $c_1= n+1+b_1+\lambda_1$.

If $c_1<  n+1+b_1+\lambda_1$, let $\sigma_1=n+1+b_1+\lambda_1-c_1$
and $\vec{\sigma}:=(\sigma_1+c_1,c_2)$. Then $\sigma_1+c_1=n+1+b_1+\lambda_1$
and, from the above proof, it follows that $S_{\negthinspace 0,\,\vb,\,\vec{\sigma}}$ is
bounded from $\lv$ to $\lt$. Notice that
     $$\frac{1}{|1-\langle w,\eta\rangle|^{c_1}}=\frac{|1-\langle w,\eta\rangle|^{\sigma_1}}{|1-\langle w,\eta\rangle|^{c_1+\sigma_1}}\leq \frac{2^{\sigma_1}}{|1-\langle w,\eta\rangle|^{\sigma_1+c_1}}.$$
Therefore, for any $f\in \lv$, we have
  $|S_{\negthinspace 0,\,\vb,\,\vc}f(z,w)|\leq 2^{\sigma_1}S_{\negthinspace 0,\,\vb,\,\vec{\sigma}}|f|(z,w),$
which implies that the operator $S_{\negthinspace 0,\,\vb,\,\vc}$ is also bounded from $\lv$ to $\lt$ in the case
$c_1<  n+1+b_1+\lambda_1$. This finishes the proof of Lemma \ref{1l3}.
\end{proof}

\begin{lemma}\label{1l4}
	Let $1=p_1<p_2\leq q_-<\infty$.
     If the parameters satisfy
     \begin{align*}
     \left\{
     \begin{aligned}
     &\alpha_1=b_1,\ \ c_1< \frac{n+1+\beta_1}{q_1}, \\
     &\alpha_2+1<p_2(b_2+1),\ \ c_2\le n+1+b_2+\frac{n+1+\beta_2}{q_2}-
     \frac{n+1+\alpha_2}{p_2},
     \end{aligned}
     \right.
     \end{align*}
     then the operator $S_{\negthinspace 0,\,\vb,\,\vc}$ is
     bounded from $\lv$ to $\lt$.
\end{lemma}

\begin{proof}
This lemma is a symmetric case of Lemma \ref{1l3}. Therefore, its proof is similar to that of Lemma \ref{1l3}
with the kernel replaced by
  $$  K_1(z,u):=\frac{1}{|1-\langle z,u \rangle|^{c_1}},\ \text{and}\ \ K_2(w,\eta)=\frac{(1-|\eta|^2)^{b_2-\alpha_2}}{|1-\langle w,\eta \rangle|^{c_2}},$$ and hence we omit the proof.
\end{proof}

\begin{lemma}\label{1l5}
	 Let $1=p_+\leq q_-<\infty$.
     If the parameters satisfy for any $i\in\{1,2\}$,
     \begin{align*}
     \left\{
     \begin{aligned}
     &\alpha_i=b_i,\\
     &c_i<\frac{n+1+\beta_i}{q_i},
     \end{aligned}
     \right.
     \end{align*}
     then the operator $S_{\negthinspace 0,\,\vb,\,\vc}$ is
     bounded from $\lv$ to $\lt$.
\end{lemma}

\begin{proof}
 Suppose $c_i\neq0$ for $i\in \{1,2\}$.
  From the fact $-(1+\beta_i)/q_i<0$ for $i\in \{1,2\}$, we deduce that there exist two negative numbers $r_1$ and $r_2$ such that, for any $i\in \{1,2\}$,
$
    -\frac{1+\beta_i}{q_i}<r_i<0.
$
  For any $i\in \{1,2\}$, let
  $\gamma_i:=-\frac{r_i}{c_i}$ and $\delta_i:=1+\frac{r_i}{c_i}.$
  Obviously, $\gamma_i+\delta_i=1$. We now define
  $ h_1(u,\eta):=1,\ h_2(z,w):=(1-|z|^2)^{r_1}(1-|w|^2)^{r_2},$
  $$ K_1(z,u):=\frac{1}{|1-\langle z,u \rangle|^{c_1}},\ \text{and}\ \ K_2(w,\eta):=\frac{1}{|1-\langle w,\eta \rangle|^{c_2}}.$$
Therefore,
$$S_{\negthinspace 0,\,\vb,\,\vc}f(z,w)=\int_{\bn}\int_{\bn}K_1(z,u)K_2(w,\eta)f(u,\eta)
\,dv_{\az_1}(u)\,dv_{\az_2}(\eta).$$
In order to use Proposition \ref{1p2} to show the boundedness of $S_{\negthinspace 0,\,\vb,\,\vc}$, we next consider
\begin{align*}
[K_1(z,u)]^{\gamma_1}[K_2(w,\eta)]^{\gamma_2}h_1(u,\eta)
    =\frac{1}{|1-\langle z,u\rangle|^{c_1\gamma_1}|1-\langle w,\eta\rangle|^{c_2\gamma_2}}.
\end{align*}
  By \eqref{1e15} and the fact $c_i\gamma=-r_i>0$ for $i\in \{1,2\}$, we know that, for any $z\in \bn$
and $u\in\bn$,
$
\frac{(1-|z|^2)^{-r_1}}{|1-\langle z,u\rangle|^{c_1\gamma_1}}\leq C
$
and, for any $w\in \bn$ and $\eta \in\bn$,
$
\frac{(1-|w|^2)^{-r_2}}{|1-\langle w,\eta\rangle|^{c_2\gamma_2}}\leq C.
$
Therefore,
$$
    \mathop{\esssup}\limits_{(u,\eta)\in \mathbb{B}_n\times \bn}[K_1(z,u)]^{\gamma_1}[K_2(w,\eta)]^{\gamma_2}h_1(u,\eta)\ls h_2(z,w).
$$
Thus, condition \eqref{1e3} holds true for the operator $S_{\negthinspace 0,\,\vb,\,\vc}$.
We next check the second condition \eqref{1e4} of Proposition \ref{1p2}. Observe that
  \begin{align*}
    \int_{\mathbb{B}_n}[K_1(z,u)]^{\delta_1 q_1}[K_2(w,\eta)]^{\delta_2 q_1}[h_2(z,w)]^{q_1}\,dv_{\beta_1}(z)=\frac{(1-|w|^2)^{r_2q_1}}{|1-\langle w,\eta \rangle|^{c_2\delta_2 q_1}}
\int_{\mathbb{B}_n}\frac{(1-|z|^2)^{r_1q_1+\beta_1}}{|1-\langle z,u \rangle|^{c_1\delta_1 q_1}}dv(z).
  \end{align*}
By the definition of $r_i$, we know that, for any $i \in \{1,2\}$,
  \begin{align}\label{1e32}
    r_iq_i+\beta_i>-1.
  \end{align}
  From the choice of $\delta_i$ and the fact that $c_i<(n+1+\beta_i)/q_i$, we deduce that, for any $i \in \{1,2\}$,
  \begin{align}\label{1e33}
    n+1+r_iq_i+\beta_i-c_i\delta_i q_i=q_i\lf(\frac{n+1+\beta_i}{q_i}-c_i\r)>0.
  \end{align}
By using this, \eqref{1e32}, and Lemma \ref{asymptotic2}(i), we infer that, for any given $u\in\bn$,
  \begin{align*}
    \int_{\mathbb{B}_n}\frac{(1-|z|^2)^{r_1q_1+\beta_1}}{|1-\langle z,u \rangle|^{c_1\delta q_1}}dv(z)\leq C,
  \end{align*}
  which, together with \eqref{1e32}, \eqref{1e33}, and Lemma \ref{asymptotic2}(i) again, further implies that
  \begin{align*}
  &\int_{\mathbb{B}_n}\left[\int_{\bn}[K_1(z,u)]^{\delta_1 q_1}[K_2(w,\eta)]^{\delta_2 q_1}[h_2(z,w)]^{q_1}\,dv_{\beta_1}(z) \right]^{q_2/q_1}\,dv_{\beta_2}(w)\\
  &\hs\ls  \int_{\mathbb{B}_n}\frac{(1-|w|^2)^{r_2q_2+\beta_2}}
    {|1-\langle w,\eta\rangle|^{c_2\delta_2 q_2}}\,dv(\eta)\ls [h_1(u,\eta)]^{q_2}.
  \end{align*}
This is just inequality \eqref{1e4} for $S_{\negthinspace 0,\,\vb,\,\vc}$. Thus,
the operator $S_{\negthinspace 0,\,\vb,\,\vc}$ satisfies all conditions of Proposition \ref{1p2}.
Then, applying Proposition \ref{1p2}, we find that $S_{\negthinspace 0,\,\vb,\,\vc}$ is
bounded from $\lv$ to $\lt$ in the case $c_i\neq0$ for any $i\in\{1,2\}$.

If $c_1=0$ and $c_2\neq0$, observe that there exists a positive
number $c_0$ such that $0<c_0<(n+1+\beta_1)/q_1$. Clearly, $|1-\langle z,u\rangle|\ls 1$
for any $z\in\bn$ and $u\in\bn$.
Thus, $1\ls \frac 1{|1-\langle z,u\rangle|^{c_0}}$ and hence
$|S_{\negthinspace 0,\,\vb,\,(0,c_2)}f(z,w)|\ls S_{\negthinspace 0,\,\vb,\,(c_0,c_2)}|f|(z,w).$
Therefore, by the above proof, we conclude that the operator $S_{\negthinspace 0,\,\vb,\,(0,c_2)}$
is also bounded from $\lv$ to $\lt$.
Similarly, if $c_1\neq 0$ and $c_2=0$, then there exists a positive
number $d_0$ such that $0<d_0<(n+1+\beta_2)/q_2$. Clearly, $|1-\langle w,\eta\rangle|\ls 1$
for any $w\in\bn$ and $\eta\in\bn$.
Thus, $1\ls \frac 1{|1-\langle w,\eta\rangle|^{d_0}}$ and hence
$|S_{\negthinspace 0,\,\vb,\,(c_1,0)}f(z,w)|\ls S_{\negthinspace 0,\,\vb,\,(c_1,d_0)}|f|(z,w).$
Thus, in this case, we also find that the operator $S_{\negthinspace 0,\,\vb,\,(c_1,0)}$
is bounded from $\lv$ to $\lt$. Furthermore, when $c_1=c_2=0$, there exist two positive
numbers $c_0$ and $d_0$ as above such that
$|S_{\negthinspace 0,\,\vb,\,(0,0)}f(z,w)|\ls S_{\negthinspace 0,\,\vb,\,(c_0,d_0)}|f|(z,w),$
which, combined with the above proof, implies that the operator $S_{\negthinspace 0,\,\vb,\,(0,0)}$
is bounded from $\lv$ to $\lt$ and hence completes the whole proof of Lemma \ref{1l5}.
\end{proof}

\begin{lemma}\label{1l6}
	Let $1=p_+\leq q_-<\infty$.
     If the parameters satisfy
     \begin{align*}
     \left\{
     \begin{aligned}
     &\alpha_1=b_1,\ \ c_1< \frac{n+1+\beta_1}{q_1},\\
     &\alpha_2<b_2,\ \ c_2\le b_2-
     \alpha_2+\frac{n+1+\beta_2}{q_2},
     \end{aligned}
     \right.
     \end{align*}
     then the operator $S_{\negthinspace 0,\,\vb,\,\vc}$ is
     bounded from $\lv$ to $\lt$.
\end{lemma}

\begin{proof}
Suppose $c_1\neq0$. By the fact $-(1+\beta_i)/q_i<0$, we know that there exist two negative numbers $r_1$ and $r_2$ such that, for any $i\in \{1,2\}$,
$
    -\frac{1+\beta_i}{q_i}<r_i<0.
$
Let
  $\gamma_1:=-\frac{r_1}{c_1}\ \text{and}\ \delta_1:=1+\frac{r_1}{c_1}.$
  Obviously, $\gamma_1+\delta_1=1$. Define
$$\lambda_2:=\frac{n+1+\beta_2}{q_2}-(n+1+\alpha_2),\ \ c_2:=n+1+b_2+\lambda_2,\ \ \text{and}\ \  \tau_2:=c_2-b_2+\alpha_2.$$
Then, $\tau_2=\frac{n+1+\beta_2}{q_2}>0$. Let
$\gamma_2:=-\frac{r_2}{\tau_2}\ \text{and}\  \delta_2:=1+\frac{r_2}{\tau_2}.$
Clearly, $\gamma_2+\delta_2=1$.
  We now define
  $ h_1(u,\eta):=1,\ h_2(z,w):=(1-|z|^2)^{r_1}(1-|w|^2)^{r_2},$
  $$ K_1(z,u):=\frac{1}{|1-\langle z,u \rangle|^{c_1}},\ \text{and}\ \ K_2(w,\eta):=\frac{(1-|\eta|^2)^{b_2-\alpha_2}}{|1-\langle w,\eta \rangle|^{c_2}}.$$
Therefore,
$$S_{\negthinspace 0,\,\vb,\,\vc}f(z,w)=\int_{\bn}\int_{\bn}K_1(z,u)K_2(w,\eta)f(u,\eta)
\,dv_{\az_1}(u)\,dv_{\az_2}(\eta).$$
In order to use Proposition \ref{1p2} to show the boundedness of $S_{\negthinspace 0,\,\vb,\,\vc}$, we next consider
  \begin{align*}
    [K_1(z,u)]^{\gamma_1}[K_2(w,\eta)]^{\gamma_2}h_1(u,\eta)=\frac{1}{|1-\langle z,u \rangle|^{c_1\gamma_1}}
 \frac{(1-|\eta|^2)^{(b_2-\alpha_2)\gamma_2}}{|1-\langle w,\eta \rangle|^{c_2\gamma_2}}.
  \end{align*}
From \eqref{1e15} and the fact that $c_1\gamma_1=-r_1>0$, we deduce that, for any $z \in \bn$ and $u\in\bn$,
  \begin{align}\label{1e34}
   \frac{(1-|z|^2)^{-r_1}}{|1-\langle z,u\rangle|^{c_1\gamma}}\leq C.
  \end{align}
Since $\tau_2=c_2-b_2+\alpha_2$ and $\gamma_2=-r_2/\tau_2$, we clearly have
$c_2\gamma_2=(b_2-\alpha_2)\gamma_2-r_2.$
From this, \eqref{1e15}, \eqref{1e16}, and the fact that $(b_2-\alpha_2)\gamma_2>0$, we know that, for any $w\in \bn$ and $\eta\in\bn$,
  \begin{align*}
    \frac{(1-|\eta|^2)^{(b_2-\alpha_2)\gamma_2}(1-|w|^2)^{-r_2}}{|1-\langle w,\eta\rangle|^{c_2\gamma_2}}&=\frac{(1-|\eta|^2)^{(b_2-\alpha_2)\gamma_2}}{|1-\langle w,\eta\rangle|^{(b_2-\alpha_2)\gamma_2}}
\frac{(1-|w|^2)^{-r_2}}{|1-\langle w,\eta\rangle|^{-r_2}}\leq C,
  \end{align*}
which, combined with \eqref{1e34}, further implies that
$$
    \mathop{\esssup}\limits_{(u,\eta)\in \mathbb{B}_n\times \bn}[K_1(z,u)]^{\gamma_1}[K_2(w,\eta)]^{\gamma_2}
h_1(u,\eta)\ls h_2(z,w).
$$
  Thus, condition \eqref{1e3} holds true for the operator $S_{\negthinspace 0,\,\vb,\,\vc}$. 
  We next check the second condition \eqref{1e4}. Observe that
  \begin{align}\label{1x1}
    &\int_{\mathbb{B}_n}[K_1(z,u)]^{\delta_1 q_1}[K_2(w,\eta)]^{\delta_2 q_1}[h_2(z,w)]^{q_1}\,dv_{\beta_1}(z)\noz\\
    &\hs=\frac{(1-|\eta|^2)^{(b_2-\alpha_2)\delta_2 q_1}(1-|w|^2)^{r_2q_1}}{|1-\langle w,\eta \rangle|^{c_2\delta_2 q_1}}\int_{\mathbb{B}_n}\frac{(1-|z|^2)^{r_1q_1+\beta_1}}{|1-\langle z,u \rangle|^{c_1\delta_1 q_1}}\,dv(z).
  \end{align}
By the definition of $r_i$, we obviously have
  \begin{align}\label{1e36}
    r_iq_i+\beta_i>-1.
  \end{align}
In addition, since $\delta_1=1+r_1/c_1$ and $c_1<(n+1+\beta_1)/q_1$, it follows that
  \begin{align*}
    n+1+r_1q_1+\beta_1-c_1\delta_1 q_1=q_1\left(\frac{n+1+\beta_1}{q_1}-c_1 \right)>0.
  \end{align*}
From this, \eqref{1e36}, and Lemma \ref{asymptotic2}(i), we deduce that, for any given $u\in\bn$,
  \begin{align}\label{1e38}
    \int_{\mathbb{B}_n}\frac{(1-|z|^2)^{r_1q_1+\beta_1}}{|1-\langle z,u \rangle|^{c_1\delta q_1}}\,dv(z)
\leq C.
  \end{align}
On the other hand, the choice of $\delta_2$ implies that
$(c_2-b_2+\alpha_2)\delta_2=\tau_2\delta_2=\frac{n+1+\beta_2}{q_2}+r_2$,
which, together with the fact $\delta_2>0$, further implies that
$n+1+r_2q_2+\beta_2-c_2\delta_2q_2=-(b_2-\alpha_2)\delta_2q_2<0.$
Thus, applying this, \eqref{1x1}, \eqref{1e36}, \eqref{1e38}, and Lemma \ref{asymptotic2}(ii), we conclude that
  \begin{align*}
  &\int_{\mathbb{B}_n}\left[\int_{\bn}[K_1(z,u)]^{\delta_1 q_1}[K_2(w,\eta)]^{\delta_2 q_1}[h_2(z,w)]^{q_1}\,dv_{\beta_1}(z) \right]^{q_2/q_1}\,dv_{\beta_2}(w)\\
  &\hs\ls (1-|\eta|^2)^{(b_2-\alpha_2)\delta_2 q_2} \int_{\mathbb{B}_n}\frac{(1-|w|^2)^{r_2q_2+\beta_2}}
    {|1-\langle w,\eta\rangle|^{c_2\delta_2 q_2}}\,dv(\eta)\ls [h_1(u,\eta)]^{q_2}.
  \end{align*}
This is just inequality \eqref{1e4} for $S_{\negthinspace 0,\,\vb,\,\vc}$. Thus,
the operator $S_{\negthinspace 0,\,\vb,\,\vc}$ satisfies all conditions of Proposition \ref{1p2}.
Then, applying Proposition \ref{1p2}, we find that $S_{\negthinspace 0,\,\vb,\,\vc}$ is
bounded from $\lv$ to $\lt$ in the case $c_1\neq0$ and $c_2= n+1+b_2+\lambda_2$.

In addition, by an argument similar to that used in the proof of Lemma \ref{1l4},
we find that the operator $S_{\negthinspace 0,\,\vb,\,\vc}$ is also bounded from $\lv$ to $\lt$ when
$c_1<(n+1+\beta_1)/q_1$ and $c_2\leq b_2-\alpha_2+(n+1+\beta_2)/q_2$, which completes the proof.
\end{proof}

\begin{lemma}\label{1l7}
	Let $1=p_+\leq q_-<\infty$.
     If the parameters satisfy
     \begin{align*}
     \left\{
     \begin{aligned}
     &\alpha_1<b_1,\ \ c_1\le b_1-
     \alpha_1+\frac{n+1+\beta_1}{q_1},\\
     &\alpha_2=b_2,\ \ c_2< \frac{n+1+\beta_2}{q_2},
     \end{aligned}
     \right.
     \end{align*}
     then the operator $S_{\negthinspace 0,\,\vb,\,\vc}$ is
     bounded from $\lv$ to $\lt$.
\end{lemma}

\begin{proof}
This lemma is a symmetric case of Lemma \ref{1l6}. Therefore, its proof is similar to that of Lemma \ref{1l6}
with the kernel replaced by
  $$ K_1(z,u):=\frac{(1-|u|^2)^{b_1-\az_1}}{|1-\langle z,u \rangle|^{c_1}}\ \text{and}\ \
K_2(w,\eta):=\frac{1}{|1-\langle w,\eta \rangle|^{c_2}},$$ and hence we omit the proof.
\end{proof}

\section{Completing the proofs of Theorems \ref{0t1}, \ref{0t2}, \ref{0t3}, and \ref{0t4} \label{s3}}

In this section, we complete the proofs of our main results by two crucial observations and putting the
pieces together. To begin with, we state the two crucial observations as follows.

\begin{observation}\label{3o1}
The boundedness of $S_{\negthinspace \va,\,\vb,\,\vc}$ implies the boundedness of $T_{\va,\,\vb,\,\vc}$.
\end{observation}

\begin{observation}\label{3o2}
The operator $S_{\negthinspace \va,\,\vb,\,\vc}$ (resp., $T_{\va,\,\vb,\,\vc}$) is bounded from $\lv$ to $\lt$
if and only if $S_{\negthinspace 0,\,\vb,\,\vc}$ (resp., $T_{0,\,\vb,\,\vc}$) is bounded from 
$\lv$ to $L^{\vec q}_{\vec \beta+\vec a \vec q}$, where $\vec \beta+\vec a \vec q:=(\beta_1+a_1q_1,\beta_2+a_2q_2)$.
\end{observation}

We now show Theorems \ref{0t1}, \ref{0t2}, \ref{0t3}, and \ref{0t4} respectively.

\begin{proof}[Proof of Theorem \ref{0t1}]
Obviously, by Observation \ref{3o1}, we know that (i) implies (ii).
Combining Lemma \ref{1l1} and Observation \ref{3o2}, we conclude that
the sufficient conditions of $\lv-\lt$ boundedness of $S_{\negthinspace \va,\,\vb,\,\vc}$ are
\begin{align}\label{3e1}
   \left\{
   \begin{aligned}
   &-q_i a_i<\beta_i+1,\ \ \alpha_i+1<p_i(b_i+1),\\
   &c_i\le n+1+a_i+b_i+\frac{n+1+\beta_i}{q_i}-
   \frac{n+1+\alpha_i}{p_i},
   \end{aligned}
   \right.
\end{align}
which implies that (iii) implies (i). On the other hand, from Observation \ref{3o2} and
Lemmas \ref{2l1} and \ref{2l7}, we infer that \eqref{3e1} are also the necessary conditions of
$\lv-\lt$ boundedness of $T_{\va,\,\vb,\,\vc}$. Thus, (ii) implies (iii) and hence Theorem \ref{0t1}
holds true.
\end{proof}

\begin{proof}[Proof of Theorem \ref{0t2}]
From Observation \ref{3o1}, it follows that (i) implies (ii).
Applying Lemma \ref{1l4} and Observation \ref{3o2}, we find that
the sufficient conditions of $\lv-\lt$ boundedness of $S_{\negthinspace \va,\,\vb,\,\vc}$ are
\begin{align}\label{3e2}
 \left\{
     \begin{aligned}
     &-q_i a_i<\beta_i+1,\ \ \alpha_1=b_1,\ \ c_1< a_1+\frac{n+1+\beta_1}{q_1}, \\
     &\alpha_2+1<p_2(b_2+1),\ \ c_2\le n+1+a_2+b_2+\frac{n+1+\beta_2}{q_2}-
     \frac{n+1+\alpha_2}{p_2}.
     \end{aligned}
     \right.
\end{align}
Therefore, (iii) implies (i). In addition, by using Observation \ref{3o2} and
Lemmas \ref{2l5} and \ref{2l11}, we conclude that \eqref{3e2} are also the necessary conditions of
$\lv-\lt$ boundedness of $T_{\va,\,\vb,\,\vc}$. Thus, (ii) implies (iii) and hence the proof of
Theorem \ref{0t1} is completed.
\end{proof}

\begin{proof}[Proof of Theorem \ref{0t3}]
As a consequence of Observation \ref{3o1}, we obtain that (i) implies (ii).
By Lemma \ref{1l3} and Observation \ref{3o2}, we know that
the sufficient conditions of $\lv-\lt$ boundedness of $S_{\negthinspace \va,\,\vb,\,\vc}$ are
\begin{align}\label{3e3}
    \left\{
     \begin{aligned}
     &-q_i a_i<\beta_i+1,\ \ \alpha_1+1<p_1(b_1+1),\ \ c_1\le n+1+a_1+b_1+\frac{n+1+\beta_1}{q_1}-
     \frac{n+1+\alpha_1}{p_1}, \\
     &\alpha_2=b_2,\ \ c_2<a_2+\frac{n+1+\beta_2}{q_2}.
     \end{aligned}
     \right.
\end{align}
This implies that (iii) implies (i). On the other hand, from Observation \ref{3o2} and
Lemmas \ref{2l2} and \ref{2l8}, we deduce that \eqref{3e3} are also the necessary conditions of
$\lv-\lt$ boundedness of $T_{\va,\,\vb,\,\vc}$. Thus, (ii) implies (iii) and hence Theorem \ref{0t3}
holds true.
\end{proof}

\begin{proof}[Proof of Theorem \ref{0t4}]
By Observation \ref{3o1}, we find that (i) implies (ii).
Applying Lemma \ref{1l1} with $p_+=1$, and Observation \ref{3o2}, we know that if
     \begin{align*}
     \left\{
     \begin{aligned}
     &-q_i a_i<\beta_i+1,\ \ \alpha_i<b_i,\\
     &c_i\le a_i+b_i-\alpha_i+\frac{n+1+\beta_i}{q_i}
     \end{aligned}
     \right.
     \end{align*}
for any $i\in\{1,2\}$, then the operator $S_{\negthinspace \va,\,\vb,\,\vc}$ is bounded from $\lv$ to $\lt$. 	
In addition, from Lemma \ref{1l5} and Observation \ref{3o2}, it follows that
\eqref{0e2} is one of the sufficient condition of the $\lv-\lt$ boundedness of $S_{\negthinspace \va,\,\vb,\,\vc}$.
By using Lemma \ref{1l6} and Observation \ref{3o2}, we conclude that
\eqref{0e3} is one of the sufficient condition of the $\lv-\lt$ boundedness of the operator
$S_{\negthinspace \va,\,\vb,\,\vc}$.
From Lemma \ref{1l7} and Observation \ref{3o2}, we infer that
\eqref{0e4} is also one of the sufficient condition of the $\lv-\lt$ boundedness
of $S_{\negthinspace \va,\,\vb,\,\vc}$. Thus, (iii) implies (i).

On the other hand, by using Observation \ref{3o2} and Lemmas \ref{2l6} and \ref{2l7},
we conclude that \eqref{0e1} is one of the necessary conditions of the
$\lv-\lt$ boundedness of $T_{\va,\,\vb,\,\vc}$.
Applying Observation \ref{3o2} and Lemmas \ref{2l6} and \ref{2l12},
we know that \eqref{0e2} is one of the necessary conditions of the
$\lv-\lt$ boundedness of $T_{\va,\,\vb,\,\vc}$.
By Observation \ref{3o2} and Lemmas \ref{2l6} and \ref{2l11},
we find that \eqref{0e3} is one of the necessary conditions of the
$\lv-\lt$ boundedness of $T_{\va,\,\vb,\,\vc}$.
Combining Observation \ref{3o2} and Lemmas \ref{2l6} and \ref{2l8},
we conclude that \eqref{0e4} is one of the necessary conditions of the
$\lv-\lt$ boundedness of $T_{\va,\,\vb,\,\vc}$.
Therefore, (ii) implies (iii) and hence the proof of Theorem \ref{0t4} is finished.
\end{proof}

\section{Multiparameter Bergman projection and Berezin transform\label{s4}}

This section is devoted to specifying two important special cases of our
main theorems. More precisely, we apply the main results to the following
weighted multiparameter Bergman projection and the weighted multiparameter Berezin transform. Let
$\vec \gamma:=(\gamma_1,\gamma_2)\in(-1,\fz)^2$, the weighted multiparameter
Bergman projection $P_{\vec \gamma}$ is defined by
\begin{align*}
P_{\vec \gamma}f(z,w):=\int_{\bn}\int_{\bn}\frac{f(u,\eta)}{(1-\langle z,u \rangle)^{n+1+\gamma_1}(1-\langle w,\eta \rangle)^{n+1+\gamma_2}}\,d\upsilon_{\gamma_1}(u)\,d\upsilon_{\gamma_2}(\eta)
\end{align*}
and the weighted multiparameter Berezin transform $B_{\vec \gamma}$ is given by
\begin{align*}
B_{\vec \gamma}f(z,w)=
\int_{\mathbb{B}_n}\int_{\mathbb{B}_n}\frac{(1-|z|^2)^{n+1+\gamma_1}(1-|w|^2)^{n+1+\gamma_2}}{|1-\langle z,u \rangle|^{2(n+1+\gamma_1)}|1-\langle w,\eta \rangle|^{2(n+1+\gamma_2)}}f(u,\eta)\,d\upsilon_{\gamma_1}(u)\,d\upsilon_{\gamma_2}(\eta).
\end{align*}
For any given $\vp:=(p_1,p_2)\in(0,\fz)^2$ and $\vec{\az}:=(\az_1,\az_2)\in(-1,\fz)^2$,
the weighted mixed-norm Bergman space
$A^{\vp}_{\vec{\az}}:=A^{\vp}(\bn\times\bn,\,d\upsilon_{\alpha_1}\times d\upsilon_{\alpha_2})$
with the same norm as $\lv$ consists of holomorphic functions $f$ in $\lv$, that is,
$$A^{\vp}_{\vec{\az}}:={ H}(\bn\times\bn)\bigcap\lv,$$
where ${H}(\bn\times\bn)$ denotes all the holomorphic functions in $\bn\times\bn$.
It is clear that $A^{\vp}_{\vec{\az}}$ is a linear subspace of $\lv$, and
we can show $P_{\vec \gamma}f\in {H}(\bn\times\bn)$ as follows.

\begin{proposition}\label{4p2'}
Let $\vec{\gamma}:=(\gamma_1,\gamma_2)\in(-1,\fz)^2$ and $f\in L^{\vec{1}}_{\vec{\gamma}}$.
Then the weighted multiparameter
Bergman projection $P_{\vec \gamma}f$ is holomorphic.
\end{proposition}

\begin{proof}
We first recall the definition of holomorphic function on the domain $\bn\times\bn\subset \mathbb{C}^{2n}$.
Indeed, a function $g:\,\bn\times\bn\to \mathbb{C}$ is holomorphic if for any $j\in \{1,\ldots,2n\}$
and, for any fixed $z_1,\ldots,z_{j-1},z_{j+1},\ldots,z_{2n}$ the function
$\zeta \mapsto g(z_1,\ldots,z_{j-1},\zeta,z_{j+1},\ldots,z_{2n})$
is analytic, in the classical one-variable sense on the set
$\{\zeta \in \mathbb{C}:(z_1,\ldots,z_{j-1},\zeta,z_{j+1},\ldots,z_{2n})\in \bn\times\bn \}.$
In other words, we have to show that $g$ is holomorphic in each variable separately.
Without loss of generality, we next assume  $j\in \{1,\ldots,n\}$.
Then given any $z,w,u,\eta \in\bn$ and $\zeta,\xi \in \mathbb{C}$, for short, let
  $$z^{(j)}(\zeta):=(z_1,\ldots,z_{j-1},\zeta,z_{j+1},\cdots,z_{n})$$
  and
   $$G_1(z,u):=\frac{1}{(1-\langle z,u\rangle)^{\gamma_1}}, \quad  G_2(w,\eta):=\frac{1}{(1-\langle w,\eta\rangle)^{\gamma_2}}.$$
  For any $j\in\{1,\ldots,n\}$ and any fixed $z_1,\ldots,z_{j-1},z_{j+1},\ldots,z_{n},w_1,\ldots,w_n$, we know that
  \begin{align}\label{1}
    &\frac{P_{\vg}f(z^{(j)}(\xi),w)-P_{\vg}f(z^{(j)}(\zeta),w)}{\xi-\zeta}\noz\\
    &\hs=\int_{\bn}\int_{\bn}f(u,\eta)G_2(w,\eta)
    \frac{G_1(z^{(j)}(\xi),u)-G_1(z^{(j)}(\zeta),u)}{\xi-\zeta}\,dv_{\gamma_1}(u)\,dv_{\gamma_2}(\eta).
  \end{align}
  Since $G_1(\cdot,u)$ is holomorphic, we have
  \begin{align*}
    &\lim\limits_{\xi \to \zeta}\left|\frac{G_1(z^{(j)}(\xi),u)-G_1(z^{(j)}(\zeta),u)}{\xi-\zeta}\right|\notag\\
    &\hs=\left|\frac{\partial G_1(z^{(j)}(\cdot),u)}{\partial z_j}(z^{(j)}(\zeta),u)\right|
    =\left|\frac{\gamma_1\overline{u_j}}{(1-\langle z^{(j)}(\zeta),u\rangle)^{\gamma_1+1}}\right|
    \leq \frac{\gamma_1}{(1-|z^{(j)}(\zeta)|)^{\gamma_1+1}}.
  \end{align*}
  It follows that for a sufficiently small $\varepsilon>0$, there exists a positive constant $C=C(z^{(j)}(\zeta))$ such that, for any $\xi$ satisfying $|\xi-\zeta|<\varepsilon$ and any $u\in\bn$,
  \begin{align}\label{3}
    \left|\frac{G_1(z^{(j)}(\xi),u)-G_1(z^{(j)}(\zeta),u)}{\xi-\zeta}\right|\leq C.
  \end{align}
It is easy to see that there exists a constant $C=C(w)$ such that, for any $\eta\in\bn$,
$|G_2(w,\eta)|\leq C.$
From this and \eqref{3}, we deduce that, there exists a constant $C=C(z^{(j)}(\zeta))C(w)$ such that, for any $\xi$ satisfying $|\xi-\zeta|<\varepsilon$, and any $u,\eta\in\bn$,
\begin{align*}
  \left|f(u,\eta)G_2(w,\eta)\frac{G_1(z^{(j)}(\xi),u)-G_1(z^{(j)}(\zeta),u)}{\xi-\zeta}\right|\leq C|f(u,\eta)|\in L^1(\bn\times \bn,dv_{\gamma_1}dv_{\gamma_2}).
\end{align*}
This, together with \eqref{1} and the dominated convergence theorem, further implies that
  \begin{align*}
    &\lim\limits_{\xi \to \zeta}\frac{P_{\vg}f(z^{(j)}(\xi),w)-P_{\vg}f(z^{(j)}(\zeta),w)}{\xi-\zeta}\\
    &=\int_{\bn}\int_{\bn}f(u,\eta)G_2(w,\eta)\lim\limits_{\xi \to \zeta}\frac{G_1(z^{(j)}(\xi),u)-G_1(z^{(j)}(\zeta),u)}{\xi-\zeta}\,dv_{\gamma_1}(u)\,dv_{\gamma_2}(\eta)\\
    &=\int_{\bn}\int_{\bn}f(u,\eta)G_2(w,\eta)\frac{\partial G_1(z^{(j)}(\cdot),u)}{\partial z_j}(z^{(j)}(\zeta),u)\,dv_{\gamma_1}(u)\,dv_{\gamma_2}(\eta),
  \end{align*}
which implies the function
$\zeta \mapsto P_{\vg}f(z^{(j)}(\zeta),w)$
is analytic and hence completes the proof.
\end{proof}

Via weighted multiparameter Bergman projection, we have the following integral representation for
functions in $A^1_{\vec{\gamma}}$, which is a extension of \cite[Theorem 2.2]{zbook05}.

\begin{proposition}\label{4p1}
If $\vec{\gamma}:=(\gamma_1,\gamma_2)\in(-1,\fz)^2$ and $f\in A^{\vec 1}_{\vec{\gamma}}$, then
\begin{align*}
f(z,w)=\int_{\bn}\int_{\bn}\frac{f(u,\eta)}{(1-\langle z,u \rangle)^{n+1+\gamma_1}(1-\langle w,\eta \rangle)^{n+1+\gamma_2}}\,d\upsilon_{\gamma_1}(u)\,d\upsilon_{\gamma_2}(\eta)
\end{align*}
for all $z\in \bn$ and $w\in\bn.$
\end{proposition}

\begin{proof}
By the Fubini theorem, we know that, for almost every $\eta \in \bn$, $f(\cdot,\eta) \in L^1(\bn,dv_{\gamma_1})$. It follows that, for almost every $\eta \in \bn$, $f(\cdot,\eta)\in A^1_{\gamma_1}$. From \cite[Theorem 2.2]{zbook05}, we deduce that, for $z \in \bn$ and almost every $\eta \in \bn$,
  \begin{align}\label{6}
    \int_{\bn}\frac{f(u,\eta)}{(1-\langle z,u\rangle)^{n+1+\gamma_1}}\,dv_{\gamma_1}(u)=f(z,\eta).
  \end{align}
  Similarly, by the Fubini theorem, we know that, for almost every $z \in \bn$, $f(z,\cdot) \in L^1(\bn,dv_{\gamma_2})$. It follows that, for almost every $z \in \bn$,$f(z,\cdot)\in A^1_{\gamma_2}$. Applying \cite[Theorem 2.2]{zbook05}, we conclude that, for $w\in \bn$ and almost every $z \in \bn$,
  \begin{align}\label{7}
    \int_{\bn}\frac{f(z,\eta)}{(1-\langle w,\eta\rangle)^{n+1+\gamma_2}}\,dv_{\gamma_2}(\eta)=f(z,w).
  \end{align}
  It is easy to see that, there is a constant $C=C(z,w)$ such that
  $$\left|\frac{f(u,\eta)}{(1-\langle z,u\rangle)^{n+1+\gamma_1}(1-\langle w,\eta\rangle)^{n+1+\gamma_2}}\right| \leq C|f(u,\eta)|\in L^1(\bn\times \bn,dv_{\gamma_1}dv_{\gamma_2}),$$
  which, together with the Fubini theorem, \eqref{6}, and \eqref{7}, implies that, for almost every $z,w \in \bn$,
  \begin{align}\label{8}
     f(z,w)=\int_{\bn}\int_{\bn}\frac{f(u,\eta)}{(1-\langle z,u\rangle)^{n+1+\gamma_1}(1-\langle w,\eta\rangle)^{n+1+\gamma_2}}\,dv_{\gamma_1}(u)\,dv_{\gamma_2}(\eta).
  \end{align}
  By Proposition 5.1, we know that the integral of \eqref{8} is a holomorphic function.
Thus, \eqref{8} holds for any $z,w\in \bn$, and hence we finish the proof.
\end{proof}

Moreover, it is clear that $P_{\vec \gamma}=T_{(0,0),(\gamma_1,\gamma_2),(n+1+\gamma_1,n+1+\gamma_2)}$
and $$B_{\vec \gamma}=S_{\negthinspace (n+1+\gamma_1,n+1+\gamma_2),\,(\gamma_1,\gamma_2),2(n+1+\gamma_1,n+1+\gamma_2)}.$$
Therefore, as special cases of Theorems \ref{0t1}, \ref{0t2}, \ref{0t3}, and \ref{0t4}, 
we immediately obtain the following characterizations
of the $\lv-\lt$ boundedness of the Berezin transform $B_{\vec \gamma}$.

\begin{proposition}\label{4p2}
Suppose $\vp:=(p_1,p_2)$, $\vq:=(q_1,q_2)\in[1,\fz)^2$, $p_+:=\max\{p_1,p_2\}$, $p_-:=\min\{p_1,p_2\}$, and $q_-:=\min\{q_1,q_2\}$.
\begin{enumerate}
\item[{\rm(i)}]
When $1<p_-\le p_+\le q_-<\fz$, then $B_{\vec \gamma}$ is bounded from $\lv$ to $\lt$ if and only if, for any $i\in\{1,2\}$,
\begin{align*}
   \left\{
   \begin{aligned}
   &\alpha_i+1<p_i(\gamma_i+1),\\
   &\frac{n+1+\alpha_i}{p_i}\le\frac{n+1+\beta_i}{q_i}.
   \end{aligned}
   \right.
\end{align*}

\item[{\rm(ii)}]
When $\vp:=(1,p_2)$ and $\vq:=(q_1,q_2)$ satisfying $1<p_2\le q_-<\fz$, then
$B_{\vec \gamma}$ is bounded from $\lv$ to $\lt$ if and only if
\begin{align*}
     \left\{
     \begin{aligned}
     &\alpha_1=\gamma_1,\ \ n+1+\gamma_1< \frac{n+1+\beta_1}{q_1}, \\
     &\alpha_2+1<p_2(\gamma_2+1),\ \ \frac{n+1+\alpha_2}{p_2}\le \frac{n+1+\beta_2}{q_2}.
     \end{aligned}
     \right.
     \end{align*}

\item[{\rm(iii)}]
When $\vp:=(p_1,1)$ and $\vq:=(q_1,q_2)$ satisfying $1<p_1\le q_-<\fz$, then
$B_{\vec \gamma}$ is bounded from $\lv$ to $\lt$ if and only if
\begin{align*}
     \left\{
     \begin{aligned}
     &\alpha_1+1<p_1(\gamma_1+1),\ \ \frac{n+1+\alpha_1}{p_1}\le \frac{n+1+\beta_1}{q_1},\\
     &\alpha_2=\gamma_2,\ \ n+1+\gamma_2<\frac{n+1+\beta_2}{q_2}.
     \end{aligned}
     \right.
     \end{align*}

\item[{\rm(iv)}]
When $\vp:=(1,1)$ and $\vq:=(q_1,q_2)\in[1,\fz)^2$, then
$B_{\vec \gamma}$ is bounded from $\lv$ to $\lt$ if and only if, for any $i\in\{1,2\}$,
\begin{align*}
     \left\{
     \begin{aligned}
     &\alpha_i<\gamma_i,\\
     &n+1+\alpha_i\le \frac{n+1+\beta_i}{q_i},
     \end{aligned}
     \right.
\end{align*}
or
\begin{align*}
     \left\{
     \begin{aligned}
     &\alpha_i=\gamma_i,\\
     &n+1+\gamma_i<\frac{n+1+\beta_i}{q_i},
     \end{aligned}
     \right.
\end{align*}
or
\begin{align*}
     \left\{
     \begin{aligned}
     &\alpha_1=\gamma_1,\ \ n+1+\gamma_1< \frac{n+1+\beta_1}{q_1},\\
     &\alpha_2<\gamma_2,\ \ n+1+\alpha_2\le \frac{n+1+\beta_2}{q_2},
     \end{aligned}
     \right.
\end{align*}
or
\begin{align*}
     \left\{
     \begin{aligned}
     &\alpha_1<\gamma_1,\ \ n+1+\alpha_1\le \frac{n+1+\beta_1}{q_1},\\
     &\alpha_2=\gamma_2,\ \ n+1+\gamma_2< \frac{n+1+\beta_2}{q_2}.
     \end{aligned}
     \right.
\end{align*}
\end{enumerate}
\end{proposition}

In addition, by Proposition \ref{4p2'} and using Theorems \ref{0t1}, \ref{0t2}, \ref{0t3}, and \ref{0t4}, we further conclude the following characterizations
of the $\lv-A^{\vq}_{\vec{\beta}}$ boundedness of the Bergman projection $P_{\vec \gamma}$.

\begin{proposition}\label{4p3}
Suppose $\vp:=(p_1,p_2)$, $\vq:=(q_1,q_2)\in[1,\fz)^2$, $p_+:=\max\{p_1,p_2\}$, $p_-:=\min\{p_1,p_2\}$, and $q_-:=\min\{q_1,q_2\}$.
\begin{enumerate}
\item[{\rm(i)}]
When $1<p_-\le p_+\le q_-<\fz$, then $P_{\vec \gamma}$ is bounded from $\lv$ to $A^{\vq}_{\vec{\beta}}$ if and only if, for any $i\in\{1,2\}$,
\begin{align*}
   \left\{
   \begin{aligned}
   &\alpha_i+1<p_i(\gamma_i+1),\\
   &\frac{n+1+\alpha_i}{p_i}\le\frac{n+1+\beta_i}{q_i}.
   \end{aligned}
   \right.
\end{align*}

\item[{\rm(ii)}]
When $\vp:=(1,p_2)$ and $\vq:=(q_1,q_2)$ satisfying $1<p_2\le q_-<\fz$, then
$P_{\vec \gamma}$ is bounded from $\lv$ to $A^{\vq}_{\vec{\beta}}$ if and only if
\begin{align*}
     \left\{
     \begin{aligned}
     &\alpha_1=\gamma_1,\ \ n+1+\gamma_1< \frac{n+1+\beta_1}{q_1}, \\
     &\alpha_2+1<p_2(\gamma_2+1),\ \ \frac{n+1+\alpha_2}{p_2}\le \frac{n+1+\beta_2}{q_2}.
     \end{aligned}
     \right.
     \end{align*}

\item[{\rm(iii)}]
When $\vp:=(p_1,1)$ and $\vq:=(q_1,q_2)$ satisfying $1<p_1\le q_-<\fz$, then
$P_{\vec \gamma}$ is bounded from $\lv$ to $A^{\vq}_{\vec{\beta}}$ if and only if
\begin{align*}
     \left\{
     \begin{aligned}
     &\alpha_1+1<p_1(\gamma_1+1),\ \ \frac{n+1+\alpha_1}{p_1}\le \frac{n+1+\beta_1}{q_1},\\
     &\alpha_2=\gamma_2,\ \ n+1+\gamma_2<\frac{n+1+\beta_2}{q_2}.
     \end{aligned}
     \right.
     \end{align*}

\item[{\rm(iv)}]
When $\vp:=(1,1)$ and $\vq:=(q_1,q_2)\in[1,\fz)^2$, then
$P_{\vec \gamma}$ is bounded from $\lv$ to $A^{\vq}_{\vec{\beta}}$ if and only if, for any $i\in\{1,2\}$,
\begin{align*}
     \left\{
     \begin{aligned}
     &\alpha_i<\gamma_i,\\
     &n+1+\alpha_i\le \frac{n+1+\beta_i}{q_i},
     \end{aligned}
     \right.
\end{align*}
or
\begin{align*}
     \left\{
     \begin{aligned}
     &\alpha_i=\gamma_i,\\
     &n+1+\gamma_i<\frac{n+1+\beta_i}{q_i},
     \end{aligned}
     \right.
\end{align*}
or
\begin{align*}
     \left\{
     \begin{aligned}
     &\alpha_1=\gamma_1,\ \ n+1+\gamma_1< \frac{n+1+\beta_1}{q_1},\\
     &\alpha_2<\gamma_2,\ \ n+1+\alpha_2\le \frac{n+1+\beta_2}{q_2},
     \end{aligned}
     \right.
\end{align*}
or
\begin{align*}
     \left\{
     \begin{aligned}
     &\alpha_1<\gamma_1,\ \ n+1+\alpha_1\le \frac{n+1+\beta_1}{q_1},\\
     &\alpha_2=\gamma_2,\ \ n+1+\gamma_2< \frac{n+1+\beta_2}{q_2}.
     \end{aligned}
     \right.
\end{align*}
\end{enumerate}
\end{proposition}

\bigskip

\noindent  \textbf{\Large {Statements and Declarations}}

\medskip
\noindent  \textbf{\large{Competing Interests:}} The authors have no relevant financial or non-financial interests to disclose.

\medskip
\noindent  \textbf{\large{Data Availability:}} Data sharing not applicable to this article as no datasets were generated or analysed during the current study.


\bigskip

\noindent Long Huang, Xiaofeng Wang and Zhicheng Zeng

\medskip

\noindent School of Mathematics and Information Science,
Key Laboratory of Mathematics and Interdisciplinary Sciences of the Guangdong Higher Education Institute,
Guangzhou University, Guangzhou, 510006, People's Republic of China

\smallskip

\noindent {\it E-mails}:
\texttt{longhuang@gzhu.edu.cn} (L. Huang)

\noindent\phantom{{\it E-mails:}}
\texttt{wxf@gzhu.edu.cn} (X. Wang)

\noindent\phantom{{\it E-mails:}}
\texttt{zhichengzeng@e.gzhu.edu.cn} (Z. Zeng)


\begin{thebibliography}{10}

\bibitem{bm16}
C. Benea and C. Muscalu, 
Multiple vector-valued inequalities via the helicoidal method, 
Anal. PDE 9 (2016), 1931-1988.

\vspace{-0.3cm}

\bibitem{bm17}
C. Benea and C. Muscalu, 
Quasi-Banach valued inequalities via the helicoidal method, 
J. Funct. Anal. 273 (2017), 1295-1353. 

\vspace{-0.3cm}

\bibitem{bm22}
C. Benea and C. Muscalu, Multiple vector-valued, mixed-norm estimates for Littlewood-Paley square functions, Publ. Mat. 66 (2022), 631-681.

\vspace{-0.3cm}

\bibitem{bp61}
A. Benedek and R. Panzone,
The space $L^p$, with mixed norm,
Duke Math. J. 28 (1961), 301-324.

\vspace{-0.3cm}

\bibitem{cs}
T. Chen and W. Sun,
Iterated and mixed weak norms with applications to geometric inequalities,
J. Geom. Anal. 30 (2020), 4268-4323.

\vspace{-0.3cm}

\bibitem{cs20}
T. Chen and W. Sun,
Extension of multilinear fractional integral operators to
linear operators on mixed-norm Lebesgue spaces,
Math. Ann. 379 (2021), 1089-1172.

\vspace{-0.3cm}

\bibitem{cs22}
T. Chen and W. Sun,
Hardy-Littlewood-Sobolev inequality on mixed-norm Lebesgue spaces,
J. Geom. Anal. 32 (2022), Paper No. 101, 43 pp.

\vspace{-0.3cm}

\bibitem{cg19}
G. Cleanthous and A. G. Georgiadis,
Mixed-norm $\alpha$-modulation spaces,
Trans. Amer. Math. Soc. 373 (2020), 3323-3356.

\vspace{-0.3cm}

\bibitem{cg22}
G. Cleanthous and A. G. Georgiadis,
Product $(\alpha_1,\alpha_2)$-modulation spaces, 
Sci. China Math. 65 (2022), 1599-1640. 

\vspace{-0.3cm}

\bibitem{cgn17}
G. Cleanthous, A. G. Georgiadis and M. Nielsen,
Anisotropic mixed-norm Hardy spaces,
J. Geom. Anal. 27 (2017), 2758-2787.

\vspace{-0.3cm}

\bibitem{cgn19}
G. Cleanthous, A. G. Georgiadis and M. Nielsen,
Molecular decomposition of anisotropic homogeneous
mixed-norm spaces with applications to the boundedness of operators,
Appl. Comput. Harmon. Anal. 47 (2019), 447-480.
\vspace{-0.3cm}

\bibitem{cgn19-1}
G. Cleanthous, A. G. Georgiadis and M. Nielsen,
Fourier multipliers on anisotropic mixed-norm spaces of distributions, 
Math. Scand. 124 (2019), 289-304.

\vspace{-0.3cm}

\bibitem{cgn19-2}
G. Cleanthous, A. G. Georgiadis and M. Nielsen, 
Molecular decomposition and Fourier multipliers for holomorphic Besov and Triebel-Lizorkin spaces, 
Monatsh. Math. 188 (2019), 467-493.

\vspace{-0.3cm}

\bibitem{dw22}
L. Ding and K. Wang, 
The $L^p-L^q$ boundedness and compactness of Bergman type operators, 
Taiwanese J. Math. 26 (2022), 713-740.

\vspace{-0.3cm}

\bibitem{fp97}
R. Fefferman and J. Pipher, 
Multiparameter operators and sharp weighted inequalities,
Amer. J. Math. 119 (1997), 337-369.

\vspace{-0.3cm}

\bibitem{fr74}
F. Forelli and W. Rudin,
Projections on spaces of holomorphic functions in balls,
Indiana Univ. Math. J. 24 (1974), 593-602.

\vspace{-0.3cm}

\bibitem{g65}
E. Gagliardo, On integral trasformations with
positive kernel, Proc. Amer. Math. Soc. 16 (1965), 429-434.

\vspace{-0.3cm}

\bibitem{gn16}
A. G. Georgiadis and M. Nielsen, 
Pseudodifferential operators on mixed-norm Besov and Triebel-Lizorkin spaces, 
Math. Nachr. 289 (2016), 2019-2036.

\vspace{-0.3cm}

\bibitem{gjn17}
A. G. Georgiadis, J. Johnsen and M. Nielsen, 
Wavelet transforms for homogeneous mixed-norm Triebel-Lizorkin spaces, 
Monatsh. Math. 183 (2017), 587-624.




\vspace{-0.3cm}

\bibitem{hcy}
L. Huang, D.-C. Chang and D. Yang,
Fourier transform of anisotropic mixed-norm Hardy spaces, 
Front. Math. China 16 (2021), 119-139.

\vspace{-0.3cm}

\bibitem{hlyy}
L. Huang, J. Liu, D. Yang and W. Yuan,
Atomic and Littlewood-Paley characterizations of anisotropic
mixed-norm Hardy spaces and their applications,
J. Geom. Anal. 29 (2019), 1991-2067.

\vspace{-0.3cm}

\bibitem{hlyy18}
L. Huang, J. Liu, D. Yang and W. Yuan,
Dual spaces of anisotropic mixed-norm Hardy spaces,
Proc. Amer. Math. Soc. 147 (2019), 1201-1215.

\vspace{-0.3cm}

\bibitem{hlyy20}
L. Huang, J. Liu, D. Yang and W. Yuan,
Identification of anisotropic mixed-norm Hardy spaces
and certain homogeneous Triebel-Lizorkin spaces,
J. Approx. Theory 258 (2020), 105459, 27 pp.

\vspace{-0.3cm}

\bibitem{hw23}
L. Huang and X. Wang,
Schur's test, Bergman-type operators and Gleason's problem on radial-angular mixed spaces, 
Electron. Res. Arch. 31 (2023), 6027-6044.

\vspace{-0.3cm}

\bibitem{hwyy23}
L. Huang, F. Weisz, D. Yang and W. Yuan,
Summability of Fourier transforms on mixed-norm Lebesgue spaces via associated Herz spaces, 
Anal. Appl. (Singap.) 21 (2023), 279-328.

\vspace{-0.3cm}

\bibitem{hyy21}
L. Huang, D. Yang and W. Yuan,
Anisotropic mixed-norm Campanato-type spaces with applications to duals of anisotropic mixed-norm Hardy spaces, 
Banach J. Math. Anal. 15 (2021), Paper No. 62, 36 pp.

\vspace{-0.3cm}

\bibitem{ku19}
H. T. Kaptano\v{g}lu and A. E. \"Ureyen, 
Singular integral operators with Bergman-Besov kernels on the ball, 
Integral Equations Operator Theory 91 (2019), Paper No. 30, 30 pp. 

\vspace{-0.3cm}

\bibitem{k79}
C. J. Kolaski,
A new look at a theorem of Forelli and Rudin,
Indiana Univ. Math. J. 28 (1979), 495-499.

\vspace{-0.3cm}

\bibitem{kz06}
O. Kures and K. Zhu, A class of integral operators on
the unit ball of $\mathbb{C}^n$,
Integral Equations Operator Theory 56 (2006), 71-82.

\vspace{-0.3cm}

\bibitem{mrs95}
D. M\"{u}ller, F. Ricci and E. M. Stein,
Marcinkiewicz multipliers and multi-parameter structure on Heisenberg (-type) groups. I, 
Invent. Math. 119 (1995), 199-233. 

\vspace{-0.3cm}

\bibitem{mrs96}
D. M\"{u}ller, F. Ricci and E. M. Stein,
Marcinkiewicz multipliers and multi-parameter structure on Heisenberg (-type) groups. II,
Math. Z. 221 (1996), 267-291.

\vspace{-0.3cm}

\bibitem{o70}
G. O. Okikiolu, 
On inequalities for integral operators, 
Glasgow Math. J. 11 (1970), 126-133.

\vspace{-0.3cm}

\bibitem{qwg23}
C. Qin, M. Wang and X. Guo, 
Weighted estimates for Forelli-Rudin type operators on the Hartogs triangle, 
Banach J. Math. Anal. 17 (2023), Paper No. 11, 33 pp.

\vspace{-0.3cm}

\bibitem{qwg23'}
C. Qin, H. Wu and X. Guo, 
$L^p$ boundedness of Forelli-Rudin type operators on the Hartogs triangle, 
J. Math. Anal. Appl. 528 (2023), Paper No. 127480, 21 pp.

\vspace{-0.3cm}

\bibitem{rs92}
F. Ricci and E. M. Stein, 
Multiparameter singular integrals and maximal functions, 
Ann. Inst. Fourier (Grenoble) 42 (1992), 637-670.

\vspace{-0.3cm}

\bibitem{rbook80}
W. Rudin, Function theory in the unit ball of $\cn$, Springer-Verlag, New York-Berlin, 1980, xiii+436 pp.

\vspace{-0.3cm}

\bibitem{s11}
J. Schur, Bemerkungen zur theorie der beschr\"{a}nkten
Bilinearformen mit unendlich vielen Ver\"{a}nderlichen,
J. Reine Angew. Math. 140 (1911), 1-28.

\vspace{-0.3cm}

\bibitem{s73}
E. M. Stein,
Singular integrals and estimates for the Cauchy-Riemann equations,
Bull. Amer. Math. Soc. 79 (1973), 440-445.

\vspace{-0.3cm}

\bibitem{s67}
R. S. Strichartz,
Multipliers on fractional Sobolev spaces,
J. Math. Mech. 16 (1967), 1031-1060.

\vspace{-0.3cm}

\bibitem{z15}
R. Zhao, Generalization of Schur's test and its application
to a class of integral operators on the unit ball of $\mathbb{C}^n$,
Integral Equations Operator Theory 82 (2015), 519-532.

\vspace{-0.3cm}

\bibitem{z23}
R. Zhao, Correction to: Generalization of Schur's test and its application to a class of
integral operators on the unit ball of $\mathbb{C}^n$, 
Integral Equations Operator Theory 95 (2023), Paper No. 18, 4 pp.

\vspace{-0.3cm}

\bibitem{zz23}
R. Zhao and L. Zhou, $L^p$-$L^q$ boundedness of Forelli-Rudin type operators on the unit 
ball of $\mathbb{C}^n$, J. Funct. Anal. 282 (2022), Paper No. 109345, 26 pp.

\vspace{-0.3cm}

\bibitem{z91}
K. Zhu, A Forelli-Rudin type theorem with applications,
Complex Variables Theory Appl. 16 (1991), 107-113.

\vspace{-0.3cm}

\bibitem{zbook05}
K. Zhu, Spaces of Holomorphic Functions in the Unit ball,
Graduate Texts in Mathematics, 226. Springer-Verlag,
New York, 2005, x+271 pp.

\vspace{-0.3cm}

\bibitem{zbook07}
K. Zhu, Operator Theory in Function Spaces, Second edition.
Mathematical Surveys and Monographs, 138. American Mathematical Society, Providence, RI, 2007. xvi+348 pp.

\vspace{-0.3cm}

\bibitem{z21}
K. Zhu, The Berezin transform and its applications,
Acta Math. Sci. Ser. B (Engl. Ed.) 41 (2021), 1839-1858.

\vspace{-0.3cm}

\bibitem{z22}
K. Zhu, Embedding and compact embedding of weighted Bergman spaces, 
Illinois J. Math. 66 (2022), 435-448.

\end{thebibliography}
\end{document}